\documentclass[onefignum,onetabnum]{siamart220329}

\usepackage{graphicx}
\usepackage{epstopdf}
\usepackage{algorithmic}
\usepackage{verbatim}
\usepackage{multirow}
\usepackage[caption=false]{subfig}
\usepackage{amsmath,booktabs,ctable,threeparttable}
\usepackage{amssymb,boxedminipage}
\usepackage{microtype}
\allowdisplaybreaks[4]

\newsiamremark{remark}{Remark}
\newsiamremark{hypothesis}{Hypothesis}
\crefname{hypothesis}{Hypothesis}{Hypotheses}
\newsiamthm{claim}{Claim}
\ifpdf
  \DeclareGraphicsExtensions{.eps,.pdf,.png,.jpg}
\else
  \DeclareGraphicsExtensions{.eps}
\fi

\title{An augmented matrix-based CJ-FEAST SVDsolver for computing a partial singular value decomposition with the singular values in a given interval\thanks{\funding{Supported in part by the National Natural Science Foundation of China (No. 12171273)}}}

\headers{An augmented matrix-based CJ-FEAST SVDSOLVER}{Z. JIA AND K. Zhang}

\author{Zhongxiao Jia\thanks{Corresponding author.
Department of Mathematical Sciences, Tsinghua
University, 100084 Beijing, China
  (\email{jiazx@tsinghua.edu.cn}).}
\and Kailiang Zhang\thanks{Department of Mathematical Sciences, Tsinghua
University, 100084 Beijing, China
  (\email{zkl18@mails.tsinghua.edu.cn}).}}
\usepackage{amsopn}
\DeclareMathOperator{\diag}{diag}

\begin{document}

\maketitle

% REQUIRED
\begin{abstract}
    The cross-product matrix-based CJ-FEAST SVDsolver
    proposed previously by the authors is shown to
    compute the left singular vector possibly much less accurately than
    the right singular vector and may be numerically backward unstable
    when a desired singular value is small.
    In this paper, an alternative
    augmented matrix-based CJ-FEAST SVDsolver is considered to
    compute the singular triplets of a large matrix $A$ with the
    singular values in an interval $[a,b]$ contained
    in the singular spectrum.
    The new CJ-FEAST SVDsolver is a subspace iteration applied to an
    approximate spectral projector of the augmented matrix
    $[0, A^T; A, 0]$ associated with the eigenvalues in $[a,b]$,
    and constructs approximate left and right
    singular subspaces with the desired singular values independently,
    onto which $A$ is projected to obtain the
    Ritz approximations to the desired singular triplets.
    Compact estimates are given for the accuracy of the approximate spectral
    projector, and a number of convergence results are established.
    The new solver is proved to be always numerically backward stable.
    A convergence comparison of the cross-product and augmented matrix-based
    CJ-FEAST SVDsolvers is made,
    and a general-purpose choice strategy between the two solvers is proposed
    for the robustness and overall efficiency. Numerical experiments confirm
    all the results.
\end{abstract}

% REQUIRED

% REQUIRED

\section{Introduction}
The singular value decomposition (SVD) of $A$ is
\begin{equation}\label{svd}
    A=U\begin{pmatrix}
        \Sigma \\
        \text{\large 0}
    \end{pmatrix}V^T
\end{equation}
with the diagonals $\sigma$ of the diagonal matrix
$\Sigma$ being the singular values
and the columns $u$ and $v$ of the orthogonal matrices
$U\in\mathbb{R}^{m\times m}$ and $V\in\mathbb{R}^{n\times n}$ being the
corresponding left and right singular vectors
of $A$; see, e.g., \cite{golub2013matrix}.
In this paper, we consider such an SVD problem: Given a large
matrix $A\in\mathbb{R}^{m\times n}$
with $m\geq n\gg 1$ and a real interval $[a,b]$ with $a>0$,
determine the $n_{sv}$ singular triplets $(\sigma,u,v)$ with
the singular values $\sigma\in [a,b]$ counting multiplicities, where
\begin{displaymath}
    \begin{cases}
        Av=\sigma u,   \\
        A^Tu=\sigma v, \\
        \left\|u\right\|=\left\|v\right\|=1.
    \end{cases}
\end{displaymath}

Write the cross-product matrix $S_C=A^TA$. Then
the eigendecomposition of $S_C=A^TA$ is
$S_C=V\Sigma^2V^T$.
The SVD of $A$ is also intimately related to the
eigendecomposition of the augmented matrix
\begin{equation}\label{aug}
    S_A=\begin{bmatrix}
        0 & A^T\\
        A & 0
    \end{bmatrix}.
\end{equation}
In the SVD \eqref{svd} of $A$, write
\begin{equation}
    U=[\underset{n}{U_n}|\underset{m-n}{\hat{U}}],
\end{equation}
and define the orthogonal matrix $Q\in \mathbb{R}^{(m+n)\times(m+n)}$ by
\begin{equation}\label{Qdef}
    Q=\frac{1}{\sqrt{2}}\begin{bmatrix}
        V & V & 0\\
        U_n & -U_n & \sqrt{2}\hat{U}
    \end{bmatrix}.
\end{equation}
Then the eigendecomposition of $S_A$ in \eqref{aug} is
\begin{equation}\label{aeigen}
    Q^TS_AQ=\diag(\Sigma,-\Sigma,\underset{m-n}{0,...,0}).
\end{equation}
We will also write the eigenvalues $\pm\sigma$ and zeros
of $S_A$ as $\lambda_i, i=1,2,\ldots,m+n$ for later
use, whose labeling
order is postponed to \Cref{sec: augmented matrix method}.

The authors in \cite{jia2022afeastsvdsolver} have
proposed an $S_C$-based Chebshev--Jackson series
FEAST (CJ-FEAST) SVDsolver,
an adaptation of the FEAST eigensolver \cite{polizzi2009density}
to the concerning SVD problem. The FEAST eigensolver
was introduced by Polizzi \cite{polizzi2009density} in 2009 and
has been developed in
\cite{gavin2018ifeast,guttel2015zolotarev,kestyn2016feast,polizzi2020feast,
tang2014feast}, and it performs on subspaces of a fixed dimension $p$, and
uses subspace
iteration \cite{golub2013matrix,saad2011numerical,stewart2001matrix}
on an approximate spectral projector associated with the eigenvalues
in a given region to generate a sequence of subspaces, onto which
the Rayleigh--Ritz projection of the original matrix or matrix pair is
realized.
However, in the $S_C$-based CJ-FEAST SVDsolver,
rather than using a numerical quadrature based rational approximation
of the contour integral of representing the spectral projector
associated with the eigenvalues $\sigma^2\in [a^2,b^2]$,
we exploit the Chebyshev--Jackson polynomial series to construct
an approximate spectral projector,
and avoid solving several shifted linear system at each
iteration as needed in the original FEAST solver.
Moreover, we can reliably estimate the number $n_{sv}$ of desired singular
triplets, and apply subspace iteration to the approximate
spectral projector to generate
an approximate right singular subspace.
The $S_C$-based CJ-FEAST SVDsolver then constructs the corresponding
approximate left singular space by premultiplying the right one with $A$,
realize the Rayleigh--Ritz projection of $A$ onto the left and right subspaces constructed,
and compute the Rayleigh--Ritz approximations to the desired singular triplets.
We have numerically observed in \cite{jia2022afeastsvdsolver} that the
$S_C$-based CJ-FEAST SVDsolver is often a few to tens times
more efficient than the contour integral-based
IFEAST \cite{gavin2018ifeast} adapted to the SVD problem
when the interval $[a,b]$ is {\em inside} the
singular spectrum and it is competitive with the latter when
the desired singular values are extreme ones.
We have theoretically argued and numerically confirmed
in \cite{jia2022afeastsvdsolver} that the CJ-FEAST SVDsolver
is more robust than contour integral based SVDsolvers.

However, as we will show,
the $S_C$-based CJ-FEAST SVDsolver may be numerically backward unstable
when a desired singular value is small.
This is because the left searching subspaces are formed by
premultiplying the right ones
with $A$ and severely filter their information on the left
singular vectors associated with small singular values.
As a consequence, the solver may compute left singular vectors
much less accurately than the right ones,
and thus  may {\em not} converge for a reasonably prescribed stopping tolerance
in finite precision arithmetic, that is, the algorithm
may be numerically backward unstable.

To overcome the above robustness deficiency
of the $S_C$-based CJ-FEAST SVDsolver,
we will exploit $S_A$ to propose a new effective CJ-FEAST SVDsolver in this
paper. In order to distinguish the two solvers,
we abbreviate the $S_C$-based CJ--FEAST SVDsolver in \cite{jia2022afeastsvdsolver} and
the $S_A$-based CJ--FEAST SVDsolver to be proposed in this paper
as the CJ-FEAST SVDsolverC and the CJ-FEAST SVDsolverA,
respectively.
Unlike the CJ-FEAST SVDsolverC, we will construct an approximation
$P$ to the spectral projector $P_{S_A}$ of $S_A$ associated with the
eigenvalues $\sigma\in [a,b]$ by the Chebyshev--Jackson series expansion.
We apply subspace iteration to such a $P$, and
generate a sequence of approximate {\em left} and {\em right} singular
subspaces corresponding to $\sigma\in [a,b]$. Precisely,
we take the upper and lower parts of iterates to {\em independently} form
approximate right and left singular subspaces, onto which $A$ is
projected to compute the Ritz approximations to the desired singular triplets.
This is a crucial difference from the CJ-FEAST SVDsolverC,
where the iterates themselves only
generate approximate right singular subspaces
and one is only able to construct
the approximate left singular subspaces by premultiplying
the right ones with $A$. Different constructions of subspaces
lead to different convergence properties of
the two CJ-FEAST SVDsolvers.

As for similarities, the two CJ-FEAST SVDsolvers
construct approximate spectral projectors using the Chebyshev--Jackson
series. We will prove that they share some similar properties.
For instance, the approximate spectral projector constructed is unconditionally
symmetric positive semi-definite (SPSD), its eigenvalues always lies in the
interval $[0,1]$,
and the strategies on degree choices of Chebyshev--Jackson
polynomial series developed in \cite{jia2022afeastsvdsolver}
can be directly adapted to the CJ-FEAST SVDsolverA.
We can estimate $n_{sv}$ by this approximate spectral projector and
Monte--Carlo methods \cite{avron2011randomized,Cortinovis2021onrandom}, as
done in \cite{jia2022afeastsvdsolver}. However,
this estimation is more costly than that in
\cite{jia2022afeastsvdsolver} as the same approximation accuracy
of $P$ for $S_A$ needs higher degree Chebyshev--Jackson series
than for $S_C$. This suggests us to estimate $n_{sv}$
using the approximate spectral projector in the CJ-FEAST SVDsolverC;
see \cite{jia2022afeastsvdsolver} for details and numerical
justifications.

As for dissimilarities, a convergence analysis of the CJ-FEAST SVDsolverA is
more involved than and quite different from that of the CJ-FEAST SVDsolverC.
For instance,
suppose that two subspaces with equal dimension are conformally
partitioned as the lower and upper parts whose dimensions are the same as
those of the given subspaces. As a necessary step, an important
problem that we must solve is: How to bound the distance
between the two upper subspaces and that between the two lower
subspaces by the distance between the original two subspaces. We establish
compact bounds on the above distances, which extend those results in
\cite{huangjia2021,huang2013,jia2003implicitly}
from the vector case, i.e., the
subspace dimension equal to one, to the general subspace case.
These bounds should have their own significance and
may find some other applications.
We will prove that the CJ-FEAST SVDsolverA always
constructs the approximate left and
right singular subspaces with similar accuracy, so that it computes
the left and right singular vectors with similar accuracy.
Therefore, the approximate
left singular vectors are (much) more accurate than those obtained by
the CJ-FEAST SVDsolverC when desired
singular values are small, which is particularly the case that $A$
is ill conditioned and some left-most singular triplets
are required. We will prove that the CJ-FFAST SVDsolverA is
always numerically backward stable
and thus fixes the potential robustness
deficiency of the CJ-FEAST SVDsolverC.

We will theoretically compare the accuracy of the two approximate spectral
projectors constructed in the two CJ-FEAST SVDsolvers,
and quantitatively show how the convergence rates of these two SVDsolvers
are closely related.
The results indicate that the CJ-FEAST SVDsolverA converges slower than the CJ-FEAST SVDsolverC for the same series degree $d$ and the subspace dimension $p$,
but it always enables us to compute small singular triplets accurately and achieve any reasonably
prescribed stopping tolerance in finite precision arithmetic.
Combining the convergence results with the computational
cost and ultimately attainable
accuracy of the two SVDsolvers, we will propose a robust
choice strategy between them in practical computations, which
guarantees that the chosen solver converges for a reasonably
stopping tolerance in finite precision arithmetic and, meanwhile, maximizes
overall efficiency.

In \Cref{sec: review previous solver},
we review the CJ-FEAST SVDsolverC, and make an analysis on its robustness
deficiency and numerical backward stability.
In \Cref{sec:framework}, we introduce an algorithmic framework of the CJ-FEAST
SVDsolverA. In \Cref{sec:pointwise}, we review the pointwise
convergence results on the Chebyshev--Jackson series expansion, which are
used later.
Then we detail the CJ-FEAST SVDsolverA in \Cref{sec: augmented matrix method}
for computing the desired $n_{sv}$ singular triplets of $A$,
and establish the accuracy estimates for the approximate spectral projector $P$
and for its eigenvalues.
In \Cref{sec:conver}, we prove a number of convergence results on
the CJ-FEAST SVDsolverA.
In \Cref{sec: comparison} we make a theoretical comparison of
the two SVDsolvers, and propose a robust choice strategy between them
in finite precision arithmetic.
In \Cref{sec: experiments}, we report numerical experiments to confirm
our results and to illustrate the robustness of the CJ-FEAST SVDsolverA.
Finally, we conclude the paper in \Cref{sec: conclusion}.

Throughout the paper, denote by $\|\cdot\|$ the 2-norm of a vector or matrix,
by $I_n$ the identity matrix of order $n$ with $n$ dropped whenever it is
clear from the context, by $e_i$ column $i$ of $I_n$,
and by $\sigma_{\max}(X)$ and $\sigma_{\min}(X)$ the largest and
smallest singular values of a matrix $X$, respectively.
All the algorithms and results apply to a complex $A$
with the transpose of a vector or matrix replaced by its
conjugate transpose.

\section{The CJ-FEAST SVDsolverC and an analysis on its convergence results}
\label{sec: review previous solver}
Given an interval $[a,b]\subset [\sigma_{\min},\|A\|]$ with $\sigma_{\min}
=\sigma_{\min}(A)$ and $a>0$,
suppose that we are interested in the singular triplets $(\sigma,u,v)$
of $A$ with all $\sigma\in [a,b]$.

For an approximate singular triplet $(\tilde{\sigma},\tilde{u},\tilde{v})$ of $A$,
its residual is
\begin{equation}\label{resnorm}
    r=r(\tilde{\sigma}, \tilde{u}, \tilde{v}):=\begin{bmatrix}
            A\tilde{v}-\tilde{\sigma}\tilde{u} \\
            A^T\tilde{u}-\tilde{\sigma}\tilde{v}
        \end{bmatrix}.
\end{equation}
Keep in mind that a numerically backward stable algorithm means that it
can make $\|r\|/\|A\|=\mathcal{O}(\epsilon_{\mathrm{mach}})$ with $\epsilon_{\mathrm{mach}}$
being the machine precision and the constant in the big $\mathcal{O}(\cdot)$
being generic, typically $10\sim 100$.

In what follows we show that the residual norm $\|r\|$ in \eqref{resnorm} may never
achieve the level $\|A\|\mathcal{O}(\epsilon_{\mathrm{mach}})$ in finite
precision arithmetic when a desired singular value $\sigma\in [a,b]$ is small,
indicating that the solver is not
numerically backward stable and may fail
for a reasonably prescribed stopping tolerance.
%We will reveal
%that such possible numerical instability is due to the possible poor accuracy
%of the left Ritz vector $\tilde{u}$ obtained by the CJ-FEAST SVDsolverC for a
%small $\sigma$.

{\sffamily The convergence results on the CJ-FEAST SVDsolverC} (cf.
Theorems 5.1--5.2 in \cite{jia2022afeastsvdsolver}):
    {\em
    Let $\mathcal{\hat{V}}^{(k)}$ and
    $\mathcal{\hat{U}}^{(k)}=A\mathcal{\hat{V}}^{(k)}$
    be the approximate right and left subspaces
    with the dimension $p\geq n_{sv}$ at iteration $k$,
    $P_k$ be the orthogonal projector onto $\mathcal{\hat{V}}^{(k)}$,
    $\gamma_1\geq \gamma_2\geq\cdots\geq\gamma_p>\gamma_{p+1}\geq \cdots\geq
    \gamma_n$ be the eigenvalues of the approximate spectral projector of $S_C$,
    and label the singular values of $A$ in the one-one correspondence
    (cf. (4.4) and Theorem 4.1 of [13]), where
    $\gamma_1,\ldots,\gamma_p$ correspond to the singular values
    $\sigma_1,\ldots,\sigma_p$.
     Write the subspace distance $\epsilon_k = \mathrm{dist} (\mathcal{\hat{V}}^{(k)}, {\mathrm{span}}\{V_p\})$,
    where the columns of $V_p$ are the right singular vectors
    of $A$ associated with the singular values $\sigma_1,\ldots,\sigma_p$.
    Assume that each desired singular value $\sigma_i$, $i=1,2,\ldots,n_{sv}$ of $A$ is simple.
    Let $(\hat{\sigma}_i^{(k)}, \hat{u}_i^{(k)},\hat{v}_i^{(k)})$
    be the Ritz approximations to $(\sigma_i,u_i,v_i),\
    i=1,2,\ldots,n_{sv}$, and define $\beta_k=\|P_kS_C(I-P_k)\|$ and
    $\delta_i^{(k)}=\min_{j\not=i}{|\sigma_i^2-(\hat{\sigma}_j^{(k)})^2|}$
    with $(\hat{\sigma}_i^{(k)})^2,\,
    i=1,2,\ldots,p$ being the Ritz values of $S_C$ with respect to
    $\mathcal{\hat{V}}^{(k)}$.
    Then for $i=1,2,\ldots,n_{sv}$ it holds that
    \begin{align}
        \sin\angle(v_i,\hat{v}_i^{(k)}) & \leq
        \sqrt{1+\frac{\beta_k^2}{(\delta_i^{(k)})^2}} \sin\angle(v_i,
        \mathcal{\hat{V}}^{(k)}),  \label{verror}   \\
        \sin\angle(u_i,\hat{u}_i^{(k)}) &\leq
        \frac{\|A\|}{\hat{\sigma}_i^{(k)}}\sin\angle(v_i,\hat{v}_i),\label{uerror}\\
        |(\hat\sigma_i^{(k)})^2-\sigma_i^2| &\leq \|A\|^2
        (3\epsilon_k^2+\epsilon_k^4 ),\label{sigmaerror}\\
      \sin\angle(v_i,
      \mathcal{\hat{V}}^{(k)})&=\mathcal{O}\biggl(\biggl(\frac{\gamma_{p+1}}
      {\gamma_i}\biggr)^k\biggr), \   \epsilon_k=
      \mathcal{O}\biggl(\biggl(\frac{\gamma_{p+1}}{\gamma_p}\biggr)^k\biggr).\label{disest}
    \end{align}
}

In finite precision arithmetic,
\eqref{disest} means that we ultimately have $\sin\angle(v_i, \mathcal{\hat{V}}^{(k)})=\mathcal{O}(\epsilon_{\mathrm{mach}}),\ i=1,2,\ldots,n_{sv}$
and $\epsilon_k=\mathcal{O}(\epsilon_{\mathrm{mach}})$.
Keep in mind these crucial points and $\beta_k\leq \|S_C\|=\|A\|^2$.
In what follows we make an analysis on the smallest attainable
size of the residual defined by \eqref{resnorm} in finite precision arithmetic.

A detailed analysis on \cite[Theorem 1.1]{jia2006} can be easily adapted to \eqref{sigmaerror},
which shows that
\begin{equation}\label{hatsigma error}
    |\hat\sigma_i^{(k)}-\sigma_i| \leq \sqrt{2}\|A\|\epsilon_k \sqrt{\epsilon_k^2 + 3}=\|A\|\mathcal{O}(\epsilon_{\rm mach}).
\end{equation}
Therefore, the CJ-FEAST SVDsolverC always computes a desired $\sigma_i$ to
the working precision, independently of its size.

Denote by $\hat{V}^{(k)}$ the right Ritz vector matrix and $\hat{\Sigma}^{(k)}$ the Ritz value matrix.
We have $(\hat{V}^{(k)})^T S_C \hat{V}^{(k)} = (\hat{\Sigma}^{(k)})^2$.
Since the residual matrix $r_{C}^{(k)}$ of the Ritz block
$((\hat{\Sigma}^{(k)})^2,\hat{V}^{(k)})$ as an approximation to the eigenblock $(\Sigma_p^2,V_p)$ of $S_C=A^TA$ satisfies
\begin{displaymath}
    (\hat{V}^{(k)})^T r_{C}^{(k)} = (\hat{V}^{(k)})^T (S_C \hat{V}^{(k)} - \hat{V}^{(k)}(\hat{\Sigma}^{(k)})^2) = 0,
\end{displaymath}
we obtain
\begin{equation}\label{resc}
    \|r_{C}^{(k)}\|= \|\hat{V}_{\perp}^{(k)}(\hat{V}_{\perp}^{(k)})^T r_{C}^{(k)} \|
    =\|\hat{V}_{\perp}^{(k)}(\hat{V}_{\perp}^{(k)})^T S_C \hat{V}^{(k)}\|
    =\| (\hat{V}_{\perp}^{(k)})^T S_C \hat{V}^{(k)} \|.
\end{equation}
Decompose $\hat{V}^{(k)}$ and $\hat{V}_{\perp}^{(k)}$ into the orthogonal direct sums of $V_p$ and $V_{p,\perp}$, respectively:
\begin{equation}\label{orthdecom}
    \hat{V}^{(k)} = V_pY_1 + V_{p,\perp}Y_{2}, \quad \hat{V}_{\perp}^{(k)} = V_p Z_1+V_{p,\perp}Z_{2}.
\end{equation}
Then $\|Y_{2}\| = \|Z_1\| = \epsilon_k$.
Substituting this relation and \eqref{orthdecom} into \eqref{resc} yields
\begin{align}
    \|r_{C}^{(k)}\|& = \|(V_pZ_1+V_{p,\perp}Z_{2})^T S_C(V_pY_1 + V_{p,\perp}Y_{2})\| \notag\\
    &=\| Z_1^T \Sigma_p^2 Y_1 + Z_{2}^T \Sigma_p^{\prime 2} Y_{2}\| \leq 2\|S_C\| \epsilon_k.\label{resSc}
\end{align}
Let $r_{i,C}^{(k)}$ be column $i$ of $r_C^{(k)},\ i=1,2,\ldots,p$.
Since $A\hat{v}_i^{(k)}=\hat{\sigma}_i^{(k)}\hat{u}_i^{(k)}$ in the CJ-FEAST SVDsolverC, from \eqref{resSc},
the ultimate SVD relative residual norm induced by \eqref{resnorm} is
\begin{equation}\label{rattainmin}
    \frac{\|r(\hat{\sigma}_i^{(k)},\hat{u}_i^{(k)},\hat{v}_i^{(k)})\|}{\|A\|}
    =\frac{\|r_{i,C}^{(k)}\|}{\hat{\sigma}_i^{(k)}\|A\|} \leq
    \frac{\|r_{C}^{(k)}\|}{\hat{\sigma}_i^{(k)}\|A\|}\leq
    \frac{2\|A\|}{\hat{\sigma}_i^{(k)}}\epsilon_k\sim
    \frac{\|A\|}{\sigma_i}\mathcal{O}(\epsilon_{\mathrm{mach}})
\end{equation}
by noticing that $\hat{\sigma}_i^{(k)}\rightarrow \sigma_i$ and
$\epsilon_k$ ultimately attains $\mathcal{O}(\epsilon_{\mathrm{mach}})$.

Since the $\|r_{i,C}^{(k)}\|$ decrease at different linear factors
for $i=1,2,\ldots,p$ and they may differ considerably,
the right-hand sides of \eqref{rattainmin} may be substantial overestimates for $\|r_{i,C}^{(k)}\|$ with $i$ smaller.
However, it is not this case in finite precision arithmetic.
Insightfully, we will show that the right-hand side
of \eqref{rattainmin} is in fact the
ultimately attainable relative residual norm of
$(\hat{\sigma}_i^{(k)},\hat{u}_i^{(k)},\hat{v}_i^{(k)})$, and
a considerably smaller one generally cannot be expected in finite
precision arithmetic, as shown below.

By the perturbation theory and residual analysis on eigenvectors (cf. \cite[p. 250]{stewart90}), %206,
for the residual $r_{i,C}^{(k)}$ of the approximate eigenpair
$((\hat{\sigma}_i^{(k)})^2,\hat{v}_i^{(k)})$ of $S_C=A^TA$, we have
\begin{equation}\label{SC posteriori error}
    \sin\angle(v_i,\hat{v}_i^{(k)}) \leq
    \frac{\|r_{i,C}^{(k)}\|}{\mathrm{gap}_{i}^{(k)}}
\end{equation}
with
$\mathrm{gap}_{i}^{(k)}=\min_{j\not=i}|(\hat{\sigma}_i^{(k)})^2-\sigma_j^2|$.

We investigate the relationship between \eqref{verror} and \eqref{SC posteriori error}.
By \eqref{hatsigma error}, and the definitions of $\delta_i^{(k)}$ and $\mathrm{gap}_{i}^{(k)}$, we ultimately have 
\begin{displaymath}
    \delta_i^{(k)}\rightarrow \min_{j\not= i, j=1,2,\ldots,p}|(\hat{\sigma}_i^{(k)})^2-\sigma_j^2| \geq \mathrm{gap}_{i}^{(k)},
\end{displaymath}
which, together with $\beta_k\leq \|A\|^2$, leads to
\begin{displaymath}
    \sqrt{1+\frac{\beta_k^2}{(\delta_i^{(k)})^2}} \sim \frac{\|A\|^2}{\mathrm{gap}_{i}^{(k)}}>1.
\end{displaymath}
Therefore, in finite precision arithmetic, \eqref{verror} means that we ultimately obtain
\begin{equation}\label{ultimatev}
    \sin\angle(v_i,\hat{v}_i^{(k)}) \leq\sqrt{1+\frac{\beta_k^2}{(\delta_i^{(k)})^2}} \mathcal{O}(\epsilon_{\mathrm{mach}})= \frac{\|A\|^2\mathcal{O}(\epsilon_{\mathrm{mach}})}{\mathrm{gap}_{i}^{(k)}}.
\end{equation}
Combining \eqref{ultimatev} with \eqref{SC posteriori error},
we ultimately have
\begin{displaymath}
    \|r_{i,C}^{(k)}\|\leq \|A\|^2\mathcal{O}(\epsilon_{\mathrm{mach}}),
\end{displaymath}
showing that the ultimately attainable relative SVD residual norm
\begin{displaymath}
  \frac{\|r(\hat{\sigma}_i^{(k)},\hat{u}_i^{(k)},\hat{v}_i^{(k)})\|}{\|A\|}
    =\frac{\|r_{i,C}^{(k)}\|}{\hat{\sigma}_i^{(k)}\|A\|}\leq \frac{\|A\|}{\sigma_i}\mathcal{O}(\epsilon_{\mathrm{mach}}),
\end{displaymath}
which indicates that whether or not the CJ-FEAST SVDsolverC is numerically
backward stable for computing $(\sigma_i,u_i,v_i)$
critically depends on the size of $\|A\|/\sigma_i$.
If the size of $\|A\|/\sigma_i$ is generic,
the solver is numerically backward stable;
if $\sigma_i$ is small relative to $\|A\|$,
the solver may not be numerically backward stable.

As a matter of fact,
the possible numerical backward instability of the CJ-FEAST SVDsolverC is due
to the possible poor accuracy of left Ritz vector $\hat{u}_i^{(k)}$.
It is known from \eqref{uerror} that
\begin{displaymath}
    \sin\angle(u_i,\hat{u}_i^{(k)})\leq \frac{\|A\|}{\hat{\sigma}_i^{(k)}}
    \sin\angle(u_i,\hat{u}_i^{(k)})\rightarrow
    \frac{\|A\|}{\sigma_i}\sin\angle(v_i,\hat{v}_i^{(k)}).
\end{displaymath}
Therefore, compared with the approximation accuracy of $\hat{v}^{(k)}$,
the error of $\hat{u}^{(k)}$  may be amplified by the multiple
$\|A\|/\sigma_i$, exactly the factor in \eqref{rattainmin}.
The ultimate attainable accuracy of $\hat{u}_i^{(k)}$ critically
depends on the size of $\|A\|/\sigma_i$ and $\hat{u}_i^{(k)}$ may
be substantially inaccurate once $\|A\|/\sigma_i$ is large,
leading to the possibly numerically backward unstable of CJ-FEAST SVDsolverC.

Actually, the possible ultimate poor accuracy of $\hat{u}_i^{(k)}$ is
expected because of the possible poor left subspace
$\mathcal{\hat{U}}^{(k)}$: Exploiting $\mathcal{\hat{U}}^{(k)} = A
\mathcal{\hat{V}}^{(k)}$
and the ultimate $\sin\angle(v_i,
\mathcal{\hat{V}}^{(k)})=\mathcal{O}(\epsilon_{\mathrm{mach}})$,
it is easily justified that
\begin{equation}\label{leftsubspaceerror}
    \sin\angle(u_i, \mathcal{\hat{U}}^{(k)})\leq\frac{\|A\|}{\sigma_i}
    \sin\angle(v_i, \mathcal{\hat{V}}^{(k)})=\frac{\|A\|}{\sigma_i}
    \mathcal{O}(\epsilon_{\mathrm{mach}}),
\end{equation}
which shows that $\mathcal{\hat{U}}^{(k)}$ is
generally much less accurate than
$\mathcal{\hat{V}}^{(k)}$ when $\|A\|/\sigma_i$ is large.

In summary, we come to conclude that the CJ-FEAST SVDsolverC may fail to
converge when requiring that $\|r(\hat{\sigma}_i^{(k)},\hat{u}_i^{(k)},
\hat{v}_i^{(k)})\|/\|A\|\leq tol$ when
\begin{equation}\label{tolfail}
    \mathcal{O}(\epsilon_{\mathrm{mach}})\leq tol<\frac{\|A\|}{\sigma_i}\mathcal{O}(\epsilon_{\mathrm{mach}})
\end{equation}
with the same generic constant, say $10$, in the two big $\mathcal{O}(\cdot)$.
Therefore, for $A$ ill conditioned,
the CJ-FEAST SVDsolverC may not work well.
This may occur if the left end $a$ of $[a,b]$ is small and there
is a $\sigma_i\in [a,b]$ close to $a$.
Numerical experiments in \Cref{sec: experiments} will confirm this assertion.

The above assertion also holds for other $S_C$-based FEAST-type
or SS-type methods, e.g., \cite{imakura2021complexmoment},
where they construct approximate right and left singular subspaces
$\mathcal{V}$ and $\mathcal{U}=A\mathcal{V}$.
Since \eqref{rattainmin} and \eqref{leftsubspaceerror} also hold for these
methods, the solvers may fail to converge for a stopping tolerance $tol$
satisfying \eqref{tolfail}.

\section{The framework of the CJ-FEAST SVDsolverA}\label{sec:framework}
Define
\begin{equation}\label{ps2}
    P_{S_A}=Q_{in}Q_{in}^T+\frac{1}{2}Q_{ab}Q_{ab}^T,
\end{equation}
where $Q_{in}$ consists of the columns of $Q$ defined by
\eqref{Qdef} corresponding to the singular values $\sigma\in (a,b)$
and $Q_{ab}$ consists of the columns of $Q$ corresponding to
$\sigma$ equal to $a$ or $b$. $P_{S_A}$ is a generalized spectral projector
of $S_A$ associated with the eigenvalues $\lambda\in [a,b]$,
and is simply called the spectral projector associated with $\lambda\in [a,b]$.
%The eigenvalues of $P_{S_A}$ are 1, $\frac{1}{2}$ and 0 when
%either $a$ or $b$ is an eigenvalue of $S_A$; they are 1 and 0 when
%neither of $a$ and $b$ are eigenvalues of $S_A$.

\cref{alg:subspace iteration} is a framework of our CJ-FEAST SVDsolverA
to be considered and developed in \Cref{sec: augmented matrix method}
and \Cref{sec:conver},
where $P$ is an approximation to $P_{S_A}$.
It is a subspace iteration on $P$ that generates the $p$-dimensional
approximate left and right subspaces $\mathcal{U}^{(k)}\subset\mathbb{R}^m$
and $\mathcal{V}^{(k)}\subset\mathbb{R}^n$, which
are formed by the lower and upper parts of the current approximate eigenspace
$\mathcal{Q}^{(k)}\subset \mathbb{R}^{m+n}$ of $P$ associated with its
$p$ dominant eigenvalues, and projects $A$ onto the left and right subspaces
to compute the $n_{sv}$ desired singular triplets of $A$.

\begin{algorithm}
    \caption{Subspace iteration on the approximate spectral projector $P$ for
    computing a partial SVD of $A$ with $\sigma\in [a,b]$.}
    \label{alg:subspace iteration}
    \begin{algorithmic}[1]
        \REQUIRE{The interval $[a,b]$, the approximate spectral projector $P$,
        a $p$-dimensional subspace $\mathcal{Q}^{(0)}$ with the dimension $p\geq n_{sv}$,
        and $k=1$.}
        \ENSURE{The $n_{sv}$ converged Ritz triplets
        $(\tilde{\sigma}^{(k)},\tilde{u}^{(k)},\tilde{v}^{(k)})$ with
        $\tilde{\sigma}^{(k)}\in [a,b]$.}
        \WHILE{not converged}
            \STATE{Form the projection subspace:
            $\mathcal{Q}^{(k)}=P\mathcal{Q}^{(k-1)}$,
                and construct the approximate right singular subspace
                $\mathcal{V}^{(k)}=[I_n,\text{\large 0}]\mathcal{Q}^{(k)}$
                and approximate left singular subspace
                $\mathcal{U}^{(k)}=[\text{\large 0}, I_m]\mathcal{Q}^{(k)}$.}
            \STATE{The Rayleigh--Ritz projection: find $p$ unit-length
             $\tilde{u}^{(k)}\in \mathcal{U}^{(k)},\tilde{v}^{(k)}\in
            \mathcal{V}^{(k)}$ and $p$ scalars
            $\tilde\sigma^{(k)}\geq 0$ that satisfy
                $A\tilde{v}^{(k)}-\tilde{\sigma}^{(k)}\tilde{u}^{(k)}\perp
                \mathcal{U}^{(k)},
                A^T\tilde{u}^{(k)}-\tilde{\sigma}^{(k)}\tilde{v}^{(k)}\perp
                \mathcal{V}^{(k)}$.}
            \STATE{Compute the residual norms of the Ritz triplets
            $(\tilde{\sigma}^{(k)},\tilde{u}^{(k)},\tilde{v}^{(k)})$ for all
            the $\tilde{\sigma}^{(k)} \in [a,b]$. Set $k\leftarrow k+1$.}
        \ENDWHILE
    \end{algorithmic}
\end{algorithm}

If $P=P_{S_A}$ defined by \eqref{ps2} and the subspace dimension $p=n_{sv}$,
then provided that no vector in the initial subspace $\mathcal{Q}^{(0)}$
is orthogonal to ${\mathrm{span}}\{Q_{in},Q_{ab}\}$,
\cref{alg:subspace iteration} finds the $n_{sv}$ desired singular triplets in one iteration
since $\mathcal{Q}^{(1)} ={\mathrm{span}}\{Q_{in},Q_{ab}\}$ and
$\mathcal{U}^{(1)},\ \mathcal{V}^{(1)}$ are the exact left and right singular
subspaces of $A$ associated with the singular values $\sigma\in[a,b]$.

\section{The Chebyshev--Jackson series expansion of a specific step function}
\label{sec:pointwise}
We review the pointwise convergence results on the
Chebyshev--Jackson series expansion established
in \cite{jia2022afeastsvdsolver},
which are needed to analyze the accuracy of an approximate spectral
projector $P$ to be constructed and the convergence of the solver.
For the interval $[a,b]\subset [-1,1]$, define the step function
\begin{equation}\label{hdef}
    h(x)=
    \begin{cases}
        1, \quad x\in (a,b),                       \\
        \frac{1}{2}, \quad x \in \{a,b\}, \\
        0, \quad x\in [-1,1]\setminus [a,b],
    \end{cases}
\end{equation}
where $h(a)=h(b)=\frac{1}{2}$ equal the means of respective left and
right limits:
\begin{displaymath}
    \frac{h(a-0)+h(a+0)}{2}=\frac{h(b-0)+h(b+0)}{2}=\frac{1}{2}.
\end{displaymath}

Suppose that $h(x)$ is approximately expanded as the Chebyshev--Jackson
polynomial series of degree $d$  \cite{jay1999electronic,rivlin1981introduction}:
\begin{equation}\label{Chebyshevseries}
    h(x)\approx \phi_d(x)=\sum_{j=0}^{d}\rho_{j,d}c_j T_{j}(x),
\end{equation}
where $T_j(x)$ is the $j$-degree Chebyshev polynomial of the first kind \cite{mason2002chebyshev}:
\begin{displaymath}
    T_0(x)=1, \ T_1(x)=x, \quad  T_{j+1}(x)=2xT_{j}(x)-T_{j-1}(x), \quad j\geq 1,
\end{displaymath}
the Fourier coefficients $c_j, j=0,1,\dots,d$, are
\begin{displaymath}
    c_j=\begin{cases}
        \frac{1}{\pi}(\arccos(a)-\arccos(b)),\quad j = 0, \\
        \frac{2}{\pi}\bigl(\frac{\sin(j\arccos(a))-\sin(j\arccos(b))}{j}\bigr),\quad j > 0,
    \end{cases}
\end{displaymath}
and the Jackson damping factors $\rho_{j,d},j=0,\dots,d$ are
\begin{displaymath}
    \rho_{j,d}=\frac{(d+2-j)\sin(\frac{\pi}{d+2})\cos(\frac{j\pi}{d+2})+
        \cos(\frac{\pi}{d+2})\sin(\frac{j\pi}{d+2})}{(d+2)\sin\frac{\pi}{d+2}}.
\end{displaymath}
For $x=\cos\theta\in [-1,1]$, by \eqref{Chebyshevseries}, define
the $2\pi$-periodic functions
\begin{align}
    &g(\theta)=h(\cos\theta)=h(x),\label{gdef} \\
    &q_{d}(\theta)=\phi_d(\cos\theta)=\phi_d(x)
    =\sum_{j=0}^{d}\rho_{j,d}c_j\cos(j\theta).
    \label{qdef}
\end{align}
The following two theorems are from \cite[Lemma 3.2, Theorem 3.3, Theorem 3.4]{jia2022afeastsvdsolver}.

\begin{theorem}\label{lem:Nonnegative}
    $\phi_d(x)\in [0,1]$ holds for $x\in [-1,1]$.
\end{theorem}

\begin{theorem}\label{lem:pointwise convergence}
    Let $\alpha=\arccos(a)>\beta=\arccos(b)$.
    For $\theta\in [0,\pi]$,
    define $\Delta_{\theta}=\min\{|\theta-\alpha|,|\theta-\beta|\}.$
    Then the following pointwise error estimates hold for $d\geq 2$:
    \begin{align*}
        |q_{d}(\theta)-g(\theta)|&\leq \frac{\pi^6}{2(d+2)^3\Delta_{\theta}^4} \ \mbox{\ for\ } \ \theta\neq \alpha,\beta, \\
        |q_d(\alpha)-g(\alpha)|  & \leq \frac{\pi^6}{2(d+2)^3}{\max}\{\frac{1}{(2\pi-2\alpha)^4},\frac{1}{(\alpha-\beta)^4}\}, \\
        |q_d(\beta)-g(\beta)|    & \leq \frac{\pi^6}{2(d+2)^3}{\max}\{\frac{1}{(2\beta)^4}, \frac{1}{(\alpha-\beta)^4}\}.
    \end{align*}
\end{theorem}

By \eqref{gdef} and \eqref{qdef},
this theorem shows that $\phi_d(x)\rightarrow h(x)$ pointwise
as $d\rightarrow\infty$ for any $x\in [-1,1]$ and the convergence rate
is at least $1/(d+2)^3$. Numerical tests in \cite{jia2022afeastsvdsolver} have
illustrated that the predicted convergence rate is the sharpest.

\section{A detailed CJ-FEAST SVDsolverA}
\label{sec: augmented matrix method}

\subsection{Approximate spectral projector and its accuracy}
We use the linear transformation $l(x)=x/\|A\|$ to map the spectrum interval $[-\|A\|,\|A\|]$ of $S_A$ to $[-1,1]$.
In applications, a rough estimate for $\|A\|$ suffices. One may run the
Golub--Kahan--Lanczos bidiagonalization
method on $A$ several steps, say $20\sim 30$, to
estimate $\|A\|$ \cite{golub2013matrix,jia2003implicitly}.
For a given $[a,b] \subset [\sigma_{\min}, \|A\|]$,
the function $h(x)$ in \eqref{hdef} becomes
\begin{displaymath}
    h(x)=
    \begin{cases}
		1,\quad x\in (l(a),l(b)),\\
        \frac{1}{2},\quad x\in \{l(a), l(b)\}, \\
		0,\quad x\in [-1,1]\setminus [l(a),l(b)].
	\end{cases}
\end{displaymath}
Define the composite function $f(x)=h(l(x))$.
Then
\begin{equation}
    f(x)=
    \begin{cases}
		1,\quad x\in (a,b),\\
        \frac{1}{2},\quad x \in \{a, b\}, \\
		0,\quad x \in [-\|A\|,\|A\|]\setminus [a,b].
	\end{cases}
\end{equation}
It follows from the above and \eqref{aeigen} that the matrix function
\begin{equation}%\label{Pdef2}
    f(S_A)=Qf(\diag(\Sigma,-\Sigma,\underset{m-n}{0,...,0}))Q^T = P_{S_A},
\end{equation}
the spectral projector defined by \eqref{ps2}. Therefore,
the eigenvalues of $P_{S_A}$ precisely correspond to
the step function $f(x)$, and $P_{S_A}$ itself is
the matrix function $f(S_A)$. This way does not
represent the spectral projector $P_{S_A}$
by a contour integral as in, e.g.,
\cite{gavin2018ifeast,kestyn2016feast,polizzi2009density,sakurai2003projection,
tang2014feast}.

\Cref{lem:pointwise convergence} proves that $\phi_d(l(x))$
pointwise converges to $f(x)$ for $x\in [-\|A\|,\|A\|]$
as $d$ increases. Naturally, we construct an approximate spectral projector
as
\begin{equation}\label{tildep}
    P=\phi_{d}(l(S_A))=\sum_{j=0}^d \rho_{j,d}c_j T_j(l(S_A)),
\end{equation}
whose eigenvector matrix is $Q$ and eigenvalues are $\phi_d(l(\pm\sigma_i))$,
$i=1,2,\dots,n$ and $\phi_d(l(0))$ with multiplicity $m-n$. Remarkably,
it is known from \Cref{lem:Nonnegative} that $P$ is
SPSD as all of its eigenvalues lie in $[0,1]$.

Next we analyze $\|P_{S_A}-P\|$, and estimate
$\phi_d(l(\pm\sigma_i))$, $i=1,2,\dots,n$ and $\phi_d(l(0))$.

\begin{theorem}\label{thm:accuracyps2}
    Given the interval $[a,b]\subset [\sigma_{\min},\|A\|]$, define
    \begin{align*}
        &\alpha=\arccos(l(a)), \quad  \beta=\arccos(l(b)), \\
        &\Delta_{il}=|\arccos(l(\sigma_{il}))-\alpha|, \quad \Delta_{ir}=|\arccos(l(\sigma_{ir}))-\beta|, \\
        &\Delta_{ol}=|\arccos(l(\sigma_{ol}))-\alpha|, \quad \Delta_{or}=|\arccos(l(\sigma_{or}))-\beta|,
    \end{align*}
    where $\sigma_{il},\ \sigma_{ir}$ and $\sigma_{ol},\ \sigma_{or}$ are
    the singular values of $A$ that are the closest to $a$ and $b$
    from the inside and outside of $[a,b]$, respectively,
    and let
    \begin{displaymath}
        \Delta_{\min}=\min\{\Delta_{il}, \Delta_{ir}, \Delta_{ol}, \Delta_{or}\}.
    \end{displaymath}
    Then
    \begin{equation}\label{Accuracy of projector2}
        \|P_{S_A}-P\| \leq \frac{\pi^6}{2(d+2)^3\Delta_{\min}^4}.
    \end{equation}
    Denote by $\mathcal{L}=\{\pm\sigma_1,\ldots,
    \pm\sigma_n,\underset{m-n}{0,...,0}\}$ the spectrum of $S_A$, suppose
    $\sigma_1,\sigma_2,\ldots,\sigma_{n_{sv}}\in [a,b]$
    with $\sigma_1,\ldots,\sigma_r\in (a,b)$ and
    the $n_{sv}-r$ ones $\sigma_{r+1},\ldots, \sigma_{n_{sv}}$ equal to
    $a$ or $b$,
    and label $\gamma_i:=\phi_d(l(\sigma_i))$, $i=1,2,\dots,n_{sv}$
    in decreasing order.
    Write the complementary set $\mathcal{L}_{n_{sv}}^c=\mathcal{L}\setminus\{\sigma_1,\ldots,\sigma_{n_{sv}}\}$,
    and label the eigenvalues $\gamma ={\phi_d(l(\lambda))}$ of $P$ for $\lambda\in\mathcal{L}_{n_{sv}}^c$ as
    $\gamma_{n_{sv}+1}\geq \gamma_{n_{sv}+2}\geq\cdots\geq\gamma_{m+n}.$
    Then if
    \begin{equation}\label{dsize2}
        d>\frac{\sqrt[3]{2}\pi^2}{\Delta_{\min}^{4/3}}-2,
    \end{equation}
    it holds that
    \begin{align}
      &  \|P_{S_A}-P\|<\frac{1}{4},\label{errorest}\\
      1\geq \gamma_1 \geq  \cdots \geq \gamma_r
        >\frac{3}{4}> \gamma_{r+1}&\geq \cdots \geq\gamma_{n_{sv}}
        >\frac{1}{4}>\gamma_{n_{sv}+1}\geq \cdots \geq\gamma_{m+n}\geq 0. \label{evtildep}
    \end{align}
\end{theorem}

\begin{proof}
    Note that the eigenvalues of $P_{S_A}$ are
    \begin{displaymath}
        \begin{cases}
            f(\sigma_{i})=h(l(\sigma_{i}))=1, \quad \sigma_{i}\in (a,b),                      \\
            f(\sigma_{i})=h(l(\sigma_{i}))=\frac{1}{2}, \quad \sigma_{i}
            \in \{a,b\}, \\
            f(\lambda)=h(l(\lambda))=0, \quad \lambda\in\mathcal{L}_{n_{sv}}^c.
        \end{cases}
    \end{displaymath}
    Then we obtain
    \begin{align*}
        \|P_{S_A}-P\| &=\|f(S_A)-\phi_d(l(S_A)) \|\\
        &=\max\{\max_{i=1,2,\ldots,n_{sv}}|h(l(\sigma_i))-\phi_d(l(\sigma_i))|,
        \max_{\lambda\in\mathcal{L}_{n_{sv}}^c}|\phi_d(l(\lambda))|\}\\
        &=\max\{\max_{i=1,2,\ldots,n_{sv}}|h(\cos(\theta_i))-\phi_d(\cos(\theta_i))|,
        \max_{\theta}|\phi_d(l(\theta))|\},
    \end{align*}
    where $\theta_i=\arccos(l(\sigma_i)),\ i=1,2,\ldots,n_{sv}$ and $\theta=\arccos(l(\lambda))$
    for $\lambda\in\mathcal{L}_{n_{sv}}^c$.
    Since
    \begin{displaymath}
        \Delta_{\min} \leq {\min}\{2\pi-2\alpha,\alpha-\beta,2\beta\},
    \end{displaymath}
    it follows from \cref{lem:pointwise convergence} that \eqref{Accuracy of projector2} holds.
    It is straightforward from \eqref{Accuracy of projector2}
    that if $d$ satisfies \eqref{dsize2} then \eqref{errorest} holds.

Since all $\gamma_i\in [0,1],\,i=1,2,\ldots,m+n$, we have
    \begin{displaymath}
        \|P_{S_A}-P\|=\max\biggl\{\max_{\sigma_{i}\in (a,b)}1-\gamma_i,
        \max_{\sigma_{i}\in\{ a, b\}}\left|\frac{1}{2}-\gamma_i\right|,
        \gamma_{n_{sv}+1}\biggr\}.
    \end{displaymath}
    which, together with \eqref{errorest}, shows that
    \begin{align*}
        0\leq 1-\gamma_i&<\frac{1}{4},\ \sigma_{i}\in (a,b),\\
        \left|\frac{1}{2}-\gamma_i\right|&<\frac{1}{4},\ \sigma_{i}\in \{a, b\},\\
        0\leq \gamma_{n_{sv}+1}&<\frac{1}{4}.
    \end{align*}
    With the labeling order of $\gamma_i, \ i=1,2,\ldots,m+n$,
    the above proves \eqref{evtildep}.
\end{proof}

\begin{remark}
    If neither of $a$ and $b$ are singular values of $A$,
    the dominant eigenvalues $\gamma_1,\ldots,\gamma_{n_{sv}}$ of $P_{S_A}$ correspond to the desired
    $\sigma_1,\ldots,\sigma_{n_{sv}}$,
    provided $\|P_{S_A}-P\|<1/2$.
\end{remark}

\subsection{The detailed CJ-FEAST SVDsolverA}
Suppose that we have determined the approximate spectral projector $P$ by \eqref{tildep}
and the subspace dimension $p\geq n_{sv}$
by the estimation approach in \cite{jia2022afeastsvdsolver}.
We apply \cref{alg:subspace iteration} to $P$, form an
approximate eigenspace of $P$ associated with its $p$ dominant eigenvalues,
and compute its orthogonal basis at each iteration.
We then take upper and lower parts of the basis to form the right
and left searching subspaces $\mathcal{V}^{(k)}$ and $\mathcal{U}^{(k)}$,
compute their orthonormal base by the thin QR decompositions,
and project $A$ onto them to compute the Ritz approximations
$(\tilde{\sigma}_{i}^{(k)},\tilde{u}_{i}^{(k)},\tilde{v}_{i}^{(k)})$
to the desired singular triplets
$(\sigma_{i},u_{i},v_{i})$, $i=1,2,\dots,n_{sv}$.
We describe the procedure as \cref{alg:Augmented-PSVD}.

\begin{algorithm}
    \caption{The CJ-FEAST SVDsolverA}
    \label{alg:Augmented-PSVD}
    \begin{algorithmic}[1]
        \REQUIRE{The interval $[a,b]$,  $c_j,
        \rho_{j,d}, j=0,\dots,d, \eta, p$, and an $(m+n)$-by-$p$ orthonormal $\tilde{Q}^{(0)}\in
        \mathbb{R}^{(m+n)\times p}$ with $p\geq n_{sv}$.}
        \ENSURE{The $n_{sv}$ converged Ritz triplets $(\tilde{\sigma}_i^{(k)}, \tilde{u}_i^{(k)},
        \tilde{v}_i^{(k)})$ with $\tilde{\sigma}_i^{(k)}\in [a,b]$.}
        \FOR{$k=1,2,\dots,$}
            \STATE{Subspace iteration: $S^{(k)}=P\tilde{Q}^{(k-1)}
            =\sum_{j=0}^d \rho_{j,d}c_j T_j(l(S_A))\tilde{Q}^{(k-1)}$.}
            \STATE{Compute the QR decomposition:
            $S^{(k)}=\tilde{Q}^{(k)}R^{(k)}$, and set
                    $Y^{(k)}=[I_n,0]\tilde{Q}^{(k)} $ and $
                    Z^{(k)}=[0,I_m]\tilde{Q}^{(k)}$.  }
            \STATE{Compute the QR decompositions:
            $Y^{(k)}=Q_{1}^{(k)}R_{1}^{(k)}$ and
            $Z^{(k)}=Q_{2}^{(k)}R_{2}^{(k)}$,
                    and take $\mathcal{V}^{(k)}={\mathrm{span}}\{Q_1^{(k)}\}$ and
                    $\mathcal{U}^{(k)}={\mathrm{span}}\{Q_2^{(k)}\}$.}
            \STATE{Compute the projection matrix:
            $\bar{A}^{(k)}=(Q_{2}^{(k)})^T A Q_{1}^{(k)}$.}
            \STATE{Compute the SVD:
            $\bar{A}^{(k)}=\bar{U}^{(k)}\tilde{\Sigma}^{(k)}(\bar{V}^{(k)})^T$
            with $\tilde{\Sigma}^{(k)} = \diag(\tilde{\sigma}_1^{(k)},\dots,\tilde{\sigma}_p^{(k)})$.}
            \STATE{Form $\tilde{U}^{(k)}=Q_{2}^{(k)}\bar{U}^{(k)}$ and
            $\tilde{V}^{(k)}=Q_{1}^{(k)}\bar{V}^{(k)}$.}
            \STATE{Select those $\tilde{\sigma}_i^{(k)}\in [a,b]$, compute the
            residual norms of the Ritz approximations $(\tilde{\sigma}_i^{(k)},
            \tilde{u}_i^{(k)},\tilde{v}_i^{(k)})$ with
            $\tilde{u}_i^{(k)}=\tilde{U}^{(k)}e_i$ and
            $\tilde{v}_i^{(k)}=\tilde{V}^{(k)}e_i$, and test convergence.}
        \ENDFOR
    \end{algorithmic}
\end{algorithm}

Next we briefly count the computational cost of one iteration of
\cref{alg:Augmented-PSVD}.
Keep in mind that the computation of $Ax$ or $A^Ty$ is
one matrix-vector product, abbreviated as MV, for given vectors $x$ and $y$.

The matrix-vector product
$S_A z$ costs two MVs for a given vector $z$:
\begin{displaymath}
    \begin{bmatrix}
        \text{\large 0} & A^T\\
        A & \text{\large 0}
    \end{bmatrix}\begin{bmatrix}
        x\\
        y
    \end{bmatrix}=\begin{bmatrix}
        A^Ty\\
        Ax
    \end{bmatrix}.
\end{displaymath}
Exploiting the three-term recurrence of Chebyshev polynomials shows that
computing $T_{1}(l(S_A)) z$ requires two MVs and $m+n$ flops and
computing $T_{j}(l(S_A))z$ needs two MVs and $2(m+n)$ flops for $j=2,\ldots,d$.
Suppose that the QR decompositions at steps 3--4 are
computed by the Gram--Schmidt procedure with
reorthogonalization, and the Matlab built-in function
{\sffamily svd}, is used to compute the SVD in step 6 of \cref{alg:Augmented-PSVD}.
We can routinely count the cost of other steps.
The cost of one iteration of \cref{alg:Augmented-PSVD} and that of
the CJ-FEAST SVDsolverC are displayed in \cref{table:Computational cost},
which indicates that,
for the same subspace dimension $p$ and the series degree $d$,
the MVs consumed by \cref{alg:Augmented-PSVD} are approximately equal
to those by the CJ-FEAST SVDsolverC and \cref{alg:Augmented-PSVD} consumes more flops than the CJ-FEAST SVDsolverC.

\begin{table}[htbp]
    \begin{center}
        \resizebox*{\textwidth}{!}{
        \begin{tabular}{|c|c|c|}
            \hline
            Solvers      & MVs      & Flops                    \\
            \hline
            CJ-FEAST SVDsolverA & $(2d+3)p$ & $4(m+n)pd+(8m+6n)p^2+21p^3+2(m+n)p$ \\
            CJ-FEAST SVDsolverC & $2(d+1)p$ & $4npd+4(m+n)p^2+21p^3+2np$ \\
            \hline
        \end{tabular}
        }
        \caption{Computational cost of one iteration of the two SVDsolvers.}
        \label{table:Computational cost}
    \end{center}
\end{table}

\section{The convergence of the CJ-FEAST SVDsolverA}\label{sec:conver}
Suppose that $p\geq n_{sv}$ and the series degree $d$ is large enough so that \eqref{errorest} and \eqref{evtildep} holds.
Since \cref{alg:Augmented-PSVD} generates the subspaces
\begin{displaymath}
    {\mathrm{span}}\{\tilde{Q}^{(k)}\}={\mathrm{span}}\{S^{(k)}\}=P{\mathrm{span}}\{\tilde{Q}^{(k-1)}\},
\end{displaymath}
we inductively obtain
\begin{equation}\label{subA}
    {\mathrm{span}}\{\tilde{Q}^{(k)}\}=P^k{\mathrm{span}}\{\tilde{Q}^{(0)}\}.
\end{equation}

Recall from \Cref{thm:accuracyps2} that the eigenvalues of $P$ are
$\gamma_i=\phi_d(l(\sigma_i)),\ i=1,2,\ldots,n_{sv}$ and
$\gamma_i=\phi_d(l(\lambda_i)),\ i=n_{sv}+1,\dots,m+n$
and they are labeled in decreasing order.
Suppose that $d$ is large enough for which
$\lambda_{n_{sv}+1},\dots,\lambda_{p}$
are positive, that is, $\lambda_{n_{sv}+1},\dots,\lambda_{p}$ are the singular
values $\sigma_{n_{sv}+1},\dots,\sigma_{p}$ of $A$.
Let $q_i$ be column $i$ of the eigenvector matrix $Q$ of $P$
with the eigenvalues $\gamma_i,\ i=1,2,\ldots,m+n$. Then
the matrix $[q_1,q_2,\ldots,q_{m+n}]$ permutes the columns of $Q$ in
\eqref{Qdef},
in which its {\em first} $p$ columns are some $p$ ones of the first $n$
columns of $Q$ in \eqref{Qdef} but the latter $m+n-p$ columns do not have the
corresponding structure in \eqref{Qdef} and the corresponding
eigenvalues are $\lambda_{p+1},\ldots \lambda_{m+n}\in \mathcal{L}\setminus
\{\sigma_1,\ldots,\sigma_p\}$.

Now we set up the following notation:
\begin{align*}
    &Q_p=[q_1,\dots, q_p],\   \  Q_{p,\perp}=[q_{p+1},\dots, q_{m+n}], \\
    &\Gamma_{p}=\diag( \gamma_1, \dots, \gamma_p), \ \ \Gamma_p^{\prime}=\diag(\gamma_{p+1},\ldots,\gamma_{m+n}),  \\
    &\Sigma_{p}=\diag(\sigma_1, \dots,\sigma_p), \  \ \Sigma_{p}^{\prime}=\diag(\lambda_{p+1},\ldots,\lambda_{m+n}),\\
    &V_p=[v_1,\dots, v_p],\   \  U_p=[u_1,\dots, u_p], \\
    &V=[V_p,V_{p,\perp}],\ \ U=[U_p,U_{p,\perp}].
\end{align*}
To establish the convergence of \cref{alg:Augmented-PSVD},
we need the following two lemmas.

\begin{lemma}
    Suppose that
        $W=[\underset{p}{W_1}|\underset{N-p}{W_2}]$ and $Z=[\underset{p}{Z_1}|\underset{N-p}{Z_2}]$
    are $N\times N$ orthogonal matrices.
    Let $\mathcal{S}_1={\mathrm{span}}\{W_1\}$ and $\mathcal{S}_2={\mathrm{span}}\{Z_1\}$.
    Then the distance ${\mathrm{dist}}(\mathcal{S}_1, \mathcal{S}_2)$
    between $\mathcal{S}_1$ and $\mathcal{S}_2$ (cf. \cite[section 2.5.3]{golub2013matrix})
   satisfies
    \begin{equation}\label{dist equation}
        {\mathrm{dist}}(\mathcal{S}_1, \mathcal{S}_2)
        = \underset{X\in \mathbb{R}^{p\times p}}{\min}\| W_1-Z_1X\|.
    \end{equation}
\end{lemma}

\begin{proof}
We have
    \begin{align*}
        \underset{X\in \mathbb{R}^{p\times p}}{\min}\| W_1-Z_1X\|&=\underset{X\in \mathbb{R}^{p\times p}}{\min}\left\|Z^T( W_1-Z_1X)\right\|\\
                                                        &=\underset{X\in \mathbb{R}^{p\times p}}{\min}\left\|\begin{bmatrix}
                                                                Z_1^TW_1 -X \\
                                                                Z_2^TW_1
                                                            \end{bmatrix}\right\|\\
                                                        &=\|Z_2^TW_1\|=
                                                        {\mathrm{dist}}(\mathcal{S}_1, \mathcal{S}_2).
    \end{align*}
\end{proof}

\begin{lemma}\label{lem:dissub}
    Suppose that $\tilde{Y},Y\in \mathbb{R}^{n\times p}$ and $\tilde{Z},Z \in \mathbb{R}^{m\times p}$ with $m,n>p$ are of full column rank, and
    $[\tilde{Y}^T,\tilde{Z}^T]^T$ and $[Y^T,Z^T]^T$ are column orthonormal. Then
    {\small
    \begin{align}
        {\mathrm{dist}}({\mathrm{span}}\{\tilde{Y}\}, {\mathrm{span}}\{Y\}) &\leq \sqrt{1+\sigma^2_{\max}(\{\tilde{Y},\tilde{Z}\})}
        \ {\mathrm{dist}}({\mathrm{span}}\{\begin{bmatrix}
            \tilde{Y}\\
            \tilde{Z}
        \end{bmatrix}\}, {\mathrm{span}}\{\begin{bmatrix}
            Y \\
            Z
        \end{bmatrix}\})\label{distV}, \\
        {\mathrm{dist}}({\mathrm{span}}\{\tilde{Z}\}, {\mathrm{span}}\{Z\}) &\leq \sqrt{1+\sigma^2_{\max}(\{\tilde{Z},\tilde{Y}\})}
        \ {\mathrm{dist}}({\mathrm{span}}\{\begin{bmatrix}
            \tilde{Y}\\
            \tilde{Z}
        \end{bmatrix}\}, {\mathrm{span}}\{\begin{bmatrix}
            Y \\
            Z
        \end{bmatrix}\})\label{distU},
    \end{align}
    }
    where $\sigma_{\max}(\{\tilde{Y},\tilde{Z}\})$ is the largest generalized singular value of
    the matrix pair $\{\tilde{Y},\tilde{Z}\}$.
\end{lemma}

\begin{proof}
Under the assumption, both $[\tilde{Y}^T,\tilde{Z}^T]^T$ and $[Y^T,Z^T]^T$
have rank $p$. Therefore, they span two subspaces with equal dimension.
According to \cite[Theorem 6.1.1]{golub2013matrix}, by the assumption on
$\tilde Y$ and $\tilde Z$, the compact generalized singular value decomposition
of the matrix pair $\{\tilde Y,\tilde Z\}$ is as follows:
There exist two column orthonormal matrices $W\in \mathbb{R}^{m\times p}$,
$G\in \mathbb{R}^{n\times p}$,
a nonsingular matrix $X\in \mathbb{R}^{p\times p}$,
    and two diagonal matrices $
        C = \diag\{\alpha_1,\dots,\alpha_p \}$ and $S =
        \diag\{\beta_1,\dots,\beta_p \}
        $
    such that
    \begin{align*}
        \tilde{Z}  = W C X^{-1} , \qquad  \tilde{Y} = GS X^{-1},\qquad C^2+S^2=I_p,\\
        1>\alpha_1\geq \alpha_2\geq \cdots \geq \alpha_p > 0, \ 0<\beta_1\leq \beta_2\leq \cdots \leq \beta_p < 1.
    \end{align*}
    Therefore, we have
    \begin{equation}\label{subeq}
    {\mathrm{span}}\{\begin{bmatrix}
        \tilde{Y} \\
        \tilde{Z}
    \end{bmatrix}\}={\mathrm{span}}\{\begin{bmatrix}
        G S  \\
        W C
    \end{bmatrix}\}.
    \end{equation}
    Since $\begin{bmatrix}
                G S  \\
                W C
            \end{bmatrix}$
    is column orthonormal,
    in terms of \eqref{dist equation} and \eqref{subeq}, we have
    \begin{align*}
        {\mathrm{dist}}({\mathrm{span}}\{\begin{bmatrix}
            G S \\
            W C
        \end{bmatrix}\}, {\mathrm{span}}\{\begin{bmatrix}
            Y \\
            Z
        \end{bmatrix}\}) &= \underset{E\in \mathbb{R}^{p\times p}}{\min}\left\| \begin{bmatrix}
            G S \\
            W C
        \end{bmatrix}-\begin{bmatrix}
            Y \\
            Z
        \end{bmatrix} E \right\| \\
        &= \underset{E\in \mathbb{R}^{p\times p}}{\min}\left\| \begin{bmatrix}
            G S-YE \\
            W C-ZE
        \end{bmatrix} \right\| \\
        &\geq \underset{E\in \mathbb{R}^{p\times p}}{\min} \left\| G S -Y E \right\|  \\
        &=\underset{E\in \mathbb{R}^{p\times p}}{\min} \left\| G-Y E S^{-1} \right\| \sigma_{\min}(S).
    \end{align*}
    Let $Y=FR$ be the QR decomposition of $Y$.
    Then
    \begin{align*}
        {\mathrm{dist}}({\mathrm{span}}\{\begin{bmatrix}
            G S \\
            W C
        \end{bmatrix}\}, {\mathrm{span}}\{\begin{bmatrix}
            Y \\
            Z
        \end{bmatrix}\}) &\geq\underset{E\in \mathbb{R}^{p\times p}}{\min} \left\| G-FR E S^{-1} \right\| \sigma_{\min}(S) \\
        &=\underset{E\in \mathbb{R}^{p\times p}}{\min} \left\| G-FE \right\| \sigma_{\min}(S) \\
        &={\mathrm{dist}}({\mathrm{span}}\{G\}, {\mathrm{span}}\{F\}) \beta_1,\\
        &={\mathrm{dist}}({\mathrm{span}}\{\tilde{Y}\}, {\mathrm{span}}\{Y\}) \beta_1.
    \end{align*}
    Since $\sqrt{1+\sigma^2_{\max}(\{\tilde{Y},\tilde{Z}\})} = \sqrt{1+(\frac{\alpha_1}{\beta_1})^2} = \frac{1}{\beta_1}$, the last relation
    proves \eqref{distV}.
    The proof of \eqref{distU} is analogous.
\end{proof}

\begin{remark}
    Exchange the positions of $\tilde{Y}$ and $Y$ and those of $\tilde{Z}$ and $Z$.
    The subspace distances in \eqref{distV} and \eqref{distU} remain the same,
    and we can obtain similar bounds, where
    $\sigma_{\max}(\{\tilde{Z},\tilde{Y}\})$ and $\sigma_{\max}(\{\tilde{Y},\tilde{Z}\})$ become
    $\sigma_{\max}(\{Z,Y\})$ and $\sigma_{\max}(\{Y,Z\})$, respectively.
    Therefore, we can replace the multiples in the two bounds by
    \begin{displaymath}
        \sqrt{1+\min\{\sigma^2_{\max}(\{\tilde{Y},\tilde{Z}\}),\sigma^2_{\max}(\{Y,Z\})\}},\
        \sqrt{1+\min\{\sigma^2_{\max}(\{\tilde{Z},\tilde{Y}\}),\sigma^2_{\max}(\{Z,Y\})\}}.
    \end{displaymath}
    This lemma generalizes \cite[Theorem 2.3]{jia2003implicitly},
    \cite[Lemma 2.3]{huang2013} and
    \cite[Lemma 3.1]{huangjia2021} from the one dimensional case to the general subspace case.
\end{remark}

Next we establish the convergence results on the approximate left
and right singular subspaces $\mathcal{U}^{(k)}$, $\mathcal{V}^{(k)}$ and the
Ritz values $\tilde \sigma_{i}^{(k)}$ obtained by \cref{alg:Augmented-PSVD}.

\begin{theorem}\label{thm:generalconver}
    Suppose that
    $\gamma_p>\gamma_{p+1}$ and $Q_p^T\tilde {Q}^{(0)}$
    is invertible. Then the subspaces \eqref{subA} generated by
    \cref{alg:Augmented-PSVD} are
    \begin{equation}\label{qk}
        \tilde Q^{(k)}=(Q_p+Q_{p,\perp}E^{(k)})(M^{(k)})^{-\frac{1}{2}}U^{(k)}
    \end{equation}
    with
    \begin{align}
        &E^{(k)}=\Gamma_p^{\prime k} Q_{p,\perp}^T\tilde{Q}^{(0)}(Q_p^T\tilde{Q}^{(0)})^{-1}\Gamma_{p}^{-k},\label{ek}\\
        &M^{(k)}=I+(E^{(k)})^T E^{(k)} \label{mkdef}
    \end{align}
    and $U^{(k)}$ being an orthogonal matrix; furthermore,
    \begin{equation}\label{normek}
        \|E^{(k)}\| \leq \biggl(\frac{\gamma_{p+1}}{\gamma_p}\biggr)^k\| E^{(0)}\|
    \end{equation}
    and the distance $\epsilon^{(k)}={\mathrm{dist}}({\mathrm{span}}\{\tilde{Q}^{(k)}\},{\mathrm{span}}\{Q_p\})$ satisfies
    \begin{equation}\label{qkdist}
        \epsilon^{(k)}=\frac{\| E^{(k)}\|}{\sqrt{1+\| E^{(k)}\|^2}} \leq  \biggl(\frac{\gamma_{p+1}}{\gamma_p}\biggr)^k\|E^{(0)}\|.
    \end{equation}
    Assume that $R_1^{(k)}$ and $R_2^{(k)}$ in Step 4 of \cref{alg:Augmented-PSVD} are nonsingular.
    Then the subspace distances
    \begin{align}
        &\epsilon_1^{(k)}:={\mathrm{dist}}({\mathrm{span}}\{V_p\},{\mathrm{span}}\{Q_1^{(k)}\}) \leq \sqrt{2}\epsilon^{(k)},\label{q1kdist}\\
        &\epsilon_2^{(k)}:={\mathrm{dist}}({\mathrm{span}}\{U_p\},{\mathrm{span}}\{Q_2^{(k)}\}) \leq \sqrt{2}\epsilon^{(k)}.\label{q2kdist}
    \end{align}
    Let $(\tilde \sigma_{i}^{(k)},\tilde u_{i}^{(k)},\tilde v_{i}^{(k)} )$
    be the $p$ Ritz approximations with
    $\tilde \sigma_{1}^{(k)}, \tilde \sigma_{2}^{(k)},\dots,\tilde
    \sigma_{p}^{(k)}$ labeled in the same order as
    $\sigma_1,\sigma_2,\dots,\sigma_p$.
    Then
    \begin{equation}\label{sigmauncond}
        |\tilde \sigma_{i}^{(k)}-\sigma_i| \leq \|A\| (6(\epsilon^{(k)})^2+4(\epsilon^{(k)})^4), \ i=1,2,...,p.
    \end{equation}
\end{theorem}

\begin{proof}
    Expand $\tilde{Q}^{(0)}$ as the orthogonal direct sum of $Q_p$ and $Q_{p,\perp}$:
    \begin{displaymath}
        \tilde{Q}^{(0)}=Q_pQ_p^T\tilde{Q}^{(0)}+Q_{p,\perp}Q_{p,\perp}^T\tilde{Q}^{(0)}
        =(Q_p+Q_{p,\perp}Q_{p,\perp}^T\tilde{Q}^{(0)}(Q_p^T\tilde{Q}^{(0)})^{-1})Q_p^T\tilde{Q}^{(0)}.
    \end{displaymath}
    Define
    \begin{equation}\label{e0}
        E^{(0)}=Q_{p,\perp}^T\tilde{Q}^{(0)}(Q_p^T\tilde{Q}^{(0)})^{-1}.
    \end{equation}
    Then
    \begin{displaymath}
        \tilde{Q}^{(0)}(Q_p^T\tilde{Q}^{(0)})^{-1}=Q_p+Q_{p,\perp}E^{(0)}.
    \end{displaymath}
    From $PQ_p=Q_p\Gamma_{p}$ and $PQ_{p,\perp}=Q_{p,\perp}\Gamma_p^{\prime}$,
    we obtain
    \begin{displaymath}
        P^k\tilde{Q}^{(0)}(Q_p^T\tilde{Q}^{(0)})^{-1}\Gamma_{p}^{-k}
        =Q_p+P^kQ_{p,\perp} E^{(0)} \Gamma_{p}^{-k}
        =Q_p+Q_{p,\perp} \Gamma_p^{\prime k} E^{(0)} \Gamma_{p}^{-k}.
    \end{displaymath}
    Write $E^{(k)}= \Gamma_p^{\prime k}  E^{(0)} \Gamma_{p}^{-k}$.
    Then it follows from \eqref{e0} that $E^{(k)}$ is the one defined by \eqref{ek}.
    Therefore,
    \begin{displaymath}
        \| E^{(k)}\|\leq \biggl(\frac{ \gamma_{p+1}}{ \gamma_p}\biggr)^k\| E^{(0)}\|\rightarrow 0,
    \end{displaymath}
    which proves \eqref{normek}.
    Since
    \begin{displaymath}
        {\mathrm{span}}\{\tilde Q^{(k)}\}=P^k{\mathrm{span}}\{\tilde{Q}^{(0)}\}={\mathrm{span}}\{Q_p+Q_{p,\perp} E^{(k)}\},
    \end{displaymath}
    the column orthonormal
    \begin{displaymath}
        \tilde Q^{(k)}=(Q_p+Q_{p,\perp} E^{(k)})( M^{(k)})^{-\frac{1}{2}}U^{(k)},
    \end{displaymath}
    where
    \begin{displaymath}
        M^{(k)}=(Q_p+Q_{p,\perp} E^{(k)})^{T}(Q_p+Q_{p,\perp} E^{(k)})=I_p+(E^{(k)})^T  E^{(k)}
    \end{displaymath}
    and $U^{(k)}$ is some orthogonal matrix, which proves \eqref{qk} and
    \eqref{mkdef}.

    By the distance definition of two same dimensional subspaces, from
    \eqref{normek} we have
    \begin{displaymath}
        \epsilon^{(k)}=
        \|Q_{p,\perp}^T\tilde{Q}^{(k)}\|=\| E^{(k)}( M^{(k)})^{-1/2}U^{(k)}\|
        =\frac{\| E^{(k)}\|}{\sqrt{1+\| E^{(k)}\|^2}}
        \leq \biggl(\frac{ \gamma_{p+1}}{ \gamma_p}\biggr)^k\| E^{(0)}\|,
    \end{displaymath}
    which proves \eqref{qkdist}.
    Therefore, under the assumption that $R_1^{(k)}$ and $R_2^{(k)}$ in Step 4 of \cref{alg:Augmented-PSVD} are nonsingular,
    since $\sigma_{\max}(\{U_p,V_p\})
    =\sigma_{\max}(\{V_p,U_p\})=1$, applying \Cref{lem:dissub}
    to $[Y^T,Z^T]^T=\tilde{Q}^{(k)}$, $\tilde{Y}:=V_p$ and $\tilde{Z}:=U_p$ yields
    \begin{align*}
        &{\mathrm{dist}}({\mathrm{span}}\{V_p\},{\mathrm{span}}\{Q_1^{(k)}\}) \leq \sqrt{2} \ {\mathrm{dist}}({\mathrm{span}}\{Q_p\},{\mathrm{span}}\{\tilde Q^{(k)}\}),\\
        &{\mathrm{dist}}({\mathrm{span}}\{U_p\},{\mathrm{span}}\{Q_2^{(k)}\}) \leq \sqrt{2} \ {\mathrm{dist}}({\mathrm{span}}\{Q_p\},{\mathrm{span}}\{\tilde Q^{(k)}\}),
    \end{align*}
    which proves \eqref{q1kdist} and \eqref{q2kdist}.

    Write the orthogonal direct sum decompositions of $Q_1^{(k)}$ and $Q_2^{(k)}$ as
    \begin{align}
        Q_1^{(k)}&=(V_p+V_{p,\perp}E_1^{(k)})(M_1^{(k)})^{-\frac{1}{2}}U_1^{(k)},\\ %\label{decomp1}
        Q_2^{(k)}&=(U_p+U_{p,\perp}E_2^{(k)})(M_2^{(k)})^{-\frac{1}{2}}U_2^{(k)},%\label{decomp2}
    \end{align}
    where $M_i^{(k)}=I+(E_i^{(k)})^T E_i^{(k)},\ i=1,2$,
    $U_i^{(k)}, \ i=1,2$ are some $p\times p$ orthogonal matrices,
    and
    \begin{equation}\label{epsi}
        \epsilon_i^{(k)}=\frac{\|E_i^{(k)}\|}{\sqrt{1+\|E_i^{(k)}\|^2}}, \ i=1,2.
    \end{equation}
    By definition, we have
    \begin{displaymath}
        A^TU_{p}=V_{p}\Sigma_p, \qquad AV_{p}=U_{p}\Sigma_p.
    \end{displaymath}
    Therefore,
    \begin{align*}
        &\|U_2^{(k)}(Q_2^{(k)})^TAQ_1^{(k)}(U_1^{(k)})^T-\Sigma_p\|\\
        &=\|(M_2^{(k)})^{-\frac{1}{2}}(U_p+U_{p,\perp}E_2^{(k)})^T A (V_p+V_{p,\perp}E_1^{(k)})(M_1^{(k)})^{-\frac{1}{2}}-\Sigma_p\|\\
        &=\|(M_2^{(k)})^{-\frac{1}{2}}(U_p+U_{p,\perp}E_2^{(k)})^T(U_{p}\Sigma_p + AV_{p,\perp}E_1^{(k)})  (M_1^{(k)})^{-\frac{1}{2}}-\Sigma_p\|\\
        &=\|(M_2^{(k)})^{-\frac{1}{2}}(\Sigma_p + (U_{p,\perp}E_2^{(k)})^T A V_{p,\perp}E_1^{(k)} )  (M_1^{(k)})^{-\frac{1}{2}} -\Sigma_p\|\\
        &\leq \|(M_2^{(k)})^{-\frac{1}{2}}\Sigma_p  (M_1^{(k)})^{-\frac{1}{2}} -\Sigma_p\|+\|(M_2^{(k)})^{-\frac{1}{2}}(U_{p,\perp}E_2^{(k)})^T A V_{p,\perp}E_1^{(k)}(M_1^{(k)})^{-\frac{1}{2}}\|.
    \end{align*}
    By \eqref{epsi} and \eqref{q1kdist}, \eqref{q2kdist}, we have
    \begin{equation}\label{error2}
        \|(M_2^{(k)})^{-\frac{1}{2}}(U_{p,\perp}E_2^{(k)})^T A V_{p,\perp}E_1^{(k)}(M_1^{(k)})^{-\frac{1}{2}}\| \leq \|A\| \epsilon_1^{(k)}\epsilon_2^{(k)}\leq 2\|A\| (\epsilon^{(k)})^2.
    \end{equation}
    Let $F_i^{(k)}=I-(M_i^{(k)})^{-\frac{1}{2}}, \ i=1,2.$
    Then
    \begin{displaymath}
        \|F_i^{(k)}\|=\|I-(M_i^{(k)})^{-\frac{1}{2}}\| = 1-\frac{1}{\sqrt{1+\|E_i^{(k)}\|^2}}\leq \frac{\|E_i^{(k)}\|^2}{1+\|E_i^{(k)}\|^2} =(\epsilon_i^{(k)})^2, \quad i=1,2.
    \end{displaymath}
    Therefore,
    \begin{align*}
        \|(M_2^{(k)})^{-\frac{1}{2}}\Sigma_p  (M_1^{(k)})^{-\frac{1}{2}} -\Sigma_p\|
        &=\|(I-F_2^{(k)})\Sigma_p  (I-F_1^{(k)}) -\Sigma_p\|\\
        &=\|-F_2^{(k)}\Sigma_p - \Sigma_p F_1^{(k)}+F_2^{(k)}\Sigma_p F_1^{(k)}\|\\
        &\leq \|A\|((\epsilon_1^{(k)})^2+(\epsilon_2^{(k)})^2+(\epsilon_1^{(k)})^2(\epsilon_2^{(k)})^2)\\
        &\leq \|A\|(4(\epsilon^{(k)})^2+4(\epsilon^{(k)})^4),
    \end{align*}
    which, together with \eqref{error2}, gives
    \begin{displaymath}
        \|U_2^{(k)}(Q_2^{(k)})^TAQ_1^{(k)}(U_1^{(k)})^T-\Sigma_p\| \leq \|A\|(6(\epsilon^{(k)})^2+4(\epsilon^{(k)})^4).
    \end{displaymath}
    According to a standard perturbation result \cite[Theorem 3.3, Chapter 3]{stewart2001matrix},
    the above relation and \eqref{qkdist} establish \eqref{sigmauncond}.
\end{proof}

\begin{remark}
    Bounds \eqref{q1kdist} and \eqref{q2kdist} indicate that the
    approximate right and left singular subspaces ${\mathrm{span}}\{Q_1^{(k)}\}$ and
    ${\mathrm{span}}\{Q_2^{(k)}\}$ have similar accuracy.
    Therefore, it is expected that
    the right and left Ritz vectors $\tilde{v}_i^{(k)}$,
    $\tilde{u}_i^{(k)}$ extracted from them have
    similar accuracy too.
\end{remark}

Next we prove that the attainable accuracy of the left and right Ritz vectors
$\tilde{u}_i^{(k)}$, $\tilde{v}_i^{(k)}$ is independent of the size of
$\sigma_i$, which is opposed to the left Ritz vectors obtained by the
CJ-FEAST SVDsolverC. As a matter of fact,
the right Ritz vectors obtained by the two SVDsolvers ultimately have
similar accuracy, but the left Ritz vectors by the CJ-FEAST SVDsolverA
are much better than the ones by the CJ-FEAST SVDsolverC
for small singular values. As a consequence,
the CJ-FEAST SVDsolverA is expected to be numerically backward stable,
independently of the size of a desired $\sigma_i$.

Define the subspace
\begin{equation}%\label{wk}
    \mathcal{W}^{(k)}={\mathrm{span}}\left\{\begin{bmatrix}
        Q_1^{(k)} & 0\\
        0 & Q_2^{(k)}
    \end{bmatrix}\right\}, k=1,\ldots.
\end{equation}
It is straightforward to justify that
\begin{equation}\label{ritzpair}
    \left(\tilde\sigma_i^{(k)},\frac{1}{\sqrt{2}}\begin{bmatrix}
        \tilde{v}_i^{(k)}\\
        \tilde{u}_i^{(k)}
    \end{bmatrix}\right),\ \left(-\tilde\sigma_i^{(k)},\frac{1}{\sqrt{2}}\begin{bmatrix}
        \tilde{v}_i^{(k)}\\
        -\tilde{u}_i^{(k)}
    \end{bmatrix}\right), i=1,2,\dots,p,
\end{equation}
are the Ritz pairs of $S_A$ with respect to $\mathcal{W}^{(k)}$.
The following theorem establishes convergence results on
the left and right Ritz vectors
$\tilde{u}_{i}^{(k)}$, $\tilde{v}_{i}^{(k)}$ as well as
new and a better
convergence result on the Ritz value $\tilde{\sigma}_{i}^{(k)}$.

\begin{theorem}\label{thm:priori}
        Let $\alpha^{(k)}=\|P^{(k)}S_A(I-P^{(k)})\|$,
        where $P^{(k)}$ is the orthogonal projector onto $\mathcal{W}^{(k)}$.
        Suppose that each singular value $\sigma_i\in [a,b]$ is simple,
        and define
        \begin{displaymath}
            \eta_i^{(k)}=\min_{j\neq i}|\sigma_i-\tilde{\sigma}_j^{(k)}|, \
            i=1,2,...,n_{sv}.
        \end{displaymath}
        Then for $i=1,2,...,n_{sv}$ it holds that
        \begin{align}
            &\sin^2\angle(u_i,\tilde{u}_{i}^{(k)})+
            \sin^2\angle(v_i,\tilde{v}_{i}^{(k)})\leq
            2\biggl(1+\frac{(\alpha^{(k)})^2}{(\eta_i^{(k)})^2}\biggr)
            \biggl(\frac{\gamma_{p+1}}{\gamma_i}\biggr)^{2k} \|E^{(0)}\|^2
            \label{tildeuv},\\
            &|\sigma_i-\tilde\sigma_{i}^{(k)}| \leq
            2\|A\|\biggl(1+\frac{(\alpha^{(k)})^2}{(\eta_i^{(k)})^2}\biggr)
            \biggl(\frac{\gamma_{p+1}}{\gamma_i}\biggr)^{2k} \|E^{(0)}\|^2
            \label{tildesigma}.
        \end{align}
\end{theorem}

\begin{proof}
    Note that
    $(\tilde\sigma_i^{(k)},\frac{1}{\sqrt{2}}\biggl[\begin{smallmatrix}
        \tilde{v}_i^{(k)}\\
        \tilde{u}_i^{(k)}
    \end{smallmatrix}\biggr]), i=1,2,\dots,n_{sv},$ are the Ritz pairs of $S_A$ with respect to $\mathcal{W}^{(k)}$. An application of
    \cite[Theorem 4.6, Proposition 4.5]{saad2011numerical} yields
    \begin{align}
        &\sin\angle(q_i,\biggl[\begin{smallmatrix}
            \tilde{v}_{i}^{(k)}\\
            \tilde{u}_{i}^{(k)}
        \end{smallmatrix}\biggr]) \leq
        \sqrt{1+\frac{(\alpha^{(k)})^2}{(\eta_i^{(k)})^2}}\sin\angle(q_i,\mathcal{W}^{(k)}) , \label{Ritz vectors} \\
        &|\sigma_i-\tilde\sigma_{i}^{(k)}|  \leq   \|S_{A}-\sigma_i I\|  \sin^2\angle(q_i,\biggl[\begin{smallmatrix}
            \tilde{v}_{i}^{(k)}\\
            \tilde{u}_{i}^{(k)}
        \end{smallmatrix}\biggr]) \leq 2 \|A\|
        \sin^2\angle(q_i,\biggl[\begin{smallmatrix}
            \tilde{v}_{i}^{(k)}\\
            \tilde{u}_{i}^{(k)}
        \end{smallmatrix}\biggr]). \label{Ritz values}
    \end{align}
    Since  span$\{\tilde{Q}^{(k)}\}\subset \mathcal{W}^{(k)}$, we have
    \begin{align*}
        \sin\angle(q_i,\mathcal{W}^{(k)}) &\leq \sin\angle(q_i, {\mathrm{span}}\{\tilde{Q}^{(k)}\})\\
            %&=\sin\angle(q_i, {\mathrm{span}}\{\tilde{Q}^{(k)}(U^{(k)})^T(M^{(k)})^{1/2}\})\\
            &=\sin\angle(q_i, {\mathrm{span}}\{Q_p+Q_{p,\perp} E^{(k)}\})
            \mbox{\ \ \ by \ \eqref{qk}}\\
            &\leq \sin\angle(q_i, q_i+Q_{p,\perp} E^{(k)}e_i)\\
            &\leq \|E^{(k)}e_i\|=\|\Gamma_p^{\prime k} E^{(0)}\Gamma_{p}^{-k}e_i\|
            =\|\Gamma_p^{\prime k} E^{(0)}\gamma_{i}^{-k}e_i\|   \\
            &\leq \biggl(\frac{ \gamma_{p+1}}{ \gamma_i}\biggr)^k \| E^{(0)}\|.
    \end{align*}
    Substituting the last inequality into \eqref{Ritz vectors} gives
    \begin{equation}\label{vecaccuracy}
        \sin\angle(q_i,\biggl[\begin{smallmatrix}
            \tilde{v}_{i}^{(k)}\\
            \tilde{u}_{i}^{(k)}
        \end{smallmatrix}\biggr]) \leq
        \sqrt{1+\frac{(\alpha^{(k)})^2}{(\eta_i^{(k)})^2}} \biggl(\frac{
        \gamma_{p+1}}{ \gamma_i}\biggr)^k \| E^{(0)}\|.
    \end{equation}
    Combining \eqref{vecaccuracy} and \eqref{Ritz values} proves \eqref{tildesigma}.
    From \cite[Theorem 2.3]{jia2003implicitly}, we have
    \begin{equation}%\label{uvac}
        \sin^2\angle(u_i,\tilde{u}_{i}^{(k)})+\sin^2\angle(v_i,\tilde{v}_{i}^{(k)})
        \leq 2\sin^2\angle(q_i, \biggl[\begin{smallmatrix}
                \tilde{v}_{i}^{(k)}\\
                \tilde{u}_{i}^{(k)}
            \end{smallmatrix}\biggr]),
    \end{equation}
    which, together with \eqref{vecaccuracy}, leads to \eqref{tildeuv}.
\end{proof}

%    Suppose that $\alpha_i^{(k)}/\delta_i^{(k)}$ is uniformly bounded from
%above with respect to $k$.
    Relations \eqref{tildeuv} and \eqref{q1kdist}, \eqref{q2kdist} show that
    $\tilde{u}_i^{(k)}$ and $\tilde{v}_i^{(k)}$ by the CJ-FEAST
    SVDsolverA have similar accuracy and each of them converges at least with
    the linear factor $\gamma_{p+1}/\gamma_i$.
    On the other hand, each $\tilde{\sigma}_i^{(k)}$ converges at the linear
    factor $(\gamma_{p+1}/\gamma_i)^2,\ i=1,2,\ldots,n_{sv}$,
    meaning that the error of $\tilde{\sigma}_i^{(k)}$ is roughly the
    error squares of $\tilde{u}_i^{(k)}$ and $\tilde{v}_i^{(k)}$ until
    $|\tilde{\sigma}_i^{(k)}-\sigma_i|\leq
    \|A\|\mathcal{O}(\epsilon_{\mathrm{mach}})$ in finite precision
    arithmetic.

    We next prove that the CJ-FEAST SVDsolverA is numerically backward stable
    independently of size of $\sigma_i$. Merge \eqref{ritzpair} for $i=1,2,\ldots,p$, and recall the notation in
    Steps 6--7 of \cref{alg:Augmented-PSVD}. We have
    \begin{displaymath}
        \frac{1}{\sqrt{2}}\begin{bmatrix}
            \tilde{V}^{(k)} & \tilde{V}^{(k)}\\
            \tilde{U}^{(k)} & -\tilde{U}^{(k)}
        \end{bmatrix}^T S_A \frac{1}{\sqrt{2}}\begin{bmatrix}
            \tilde{V}^{(k)} & \tilde{V}^{(k)}\\
            \tilde{U}^{(k)} & -\tilde{U}^{(k)}
        \end{bmatrix} = \begin{bmatrix}
            \tilde{\Sigma}^{(k)} & \\
            & -\tilde{\Sigma}^{(k)}
        \end{bmatrix}.
    \end{displaymath}
    Then using the proof approach to estimating $\|r_C^{(k)}\|$ in
    \Cref{sec: review previous solver},
    we can prove that
    the residual $r_{A}^{(k)}$ of the Ritz block
    \begin{displaymath}
        \biggl(\begin{bmatrix}
            \tilde{\Sigma}^{(k)} & \\
            & -\tilde{\Sigma}^{(k)}
        \end{bmatrix}, \frac{1}{\sqrt{2}}\begin{bmatrix}
        \tilde{V}^{(k)} & \tilde{V}^{(k)}\\
            \tilde{U}^{(k)} & -\tilde{U}^{(k)}
        \end{bmatrix}\biggr)
    \end{displaymath}
    as an approximation to the eigenblock
    \begin{displaymath}
        \biggl(\begin{bmatrix}
        \Sigma_p & \\
            & -\Sigma_p
        \end{bmatrix}, \frac{1}{\sqrt{2}}\begin{bmatrix}
        V_p & V_p \\
        U_p & -U_p
        \end{bmatrix}\biggr)
    \end{displaymath}
    of $S_A$ satisfies
    \begin{equation}\label{resSA}
        \|r_{A}^{(k)}\| \leq 2\|S_A\| {\mathrm{dist}}({\mathrm{span}}\{\bigl[\begin{smallmatrix}
            \tilde{V}^{(k)} & \tilde{V}^{(k)}\\
                \tilde{U}^{(k)} & -\tilde{U}^{(k)}
        \end{smallmatrix}\bigr]\}, {\mathrm{span}}\{\bigl[\begin{smallmatrix}
            V_p & V_p \\
            U_p & -U_p
        \end{smallmatrix}\bigr]\}).
    \end{equation}
    On the other hand, we obtain
    \begin{align*}
        &{\mathrm{dist}}({\mathrm{span}}\{\bigl[\begin{smallmatrix}
            \tilde{V}^{(k)} & \tilde{V}^{(k)}\\
                \tilde{U}^{(k)} & -\tilde{U}^{(k)}
        \end{smallmatrix}\bigr]\}, {\mathrm{span}}\{\bigl[\begin{smallmatrix}
            V_p & V_p \\
            U_p & -U_p
        \end{smallmatrix}\bigr]\}) \\
        &= {\mathrm{dist}}({\mathrm{span}}\{\bigl[\begin{smallmatrix}
            \tilde{V}^{(k)} & \\
             & \tilde{U}^{(k)}
        \end{smallmatrix}\bigr]\}, {\mathrm{span}}\{\bigl[\begin{smallmatrix}
            V_p & \\
             & U_p
        \end{smallmatrix}\bigr]\})  \\
        &=\max\{\mathrm{dist}({\mathrm{span}}\{\tilde{V}^{(k)}\},
        {\mathrm{span}}\{V_p\}), \mathrm{dist}(\mathrm{span}\{\tilde{U}^{(k)}\},
        {\mathrm{span}}\{U_p\})\} \\
        &\leq \sqrt{2}\epsilon^{(k)},
    \end{align*}
    where the last inequality follows from \eqref{q1kdist} and \eqref{q2kdist}.
Let $r_{i,A}^{(k)}$ be the column $i$ of $r_A^{(k)},i=1,2,\ldots,p$.
Therefore, it follows from \eqref{resnorm}, \eqref{resSA}
and $\|S_A\|=\|A\|$ that the SVD residual norm
    \begin{displaymath}
        \|r(\tilde{\sigma}_i^{(k)},\tilde{u}_i^{(k)},\tilde{v}_i^{(k)})\|
        =\sqrt{2}\|r_{i,A}^{(k)}\| \leq \sqrt{2}\|r_{A}^{(k)}\| \leq 4
        \|A\|\epsilon^{(k)},
    \end{displaymath}
    indicating that
    the CJ-FEAST SVDsolverA is always numerically backward
    stable for computing any singular triplet of $A$ as
    $\epsilon^{(k)}=\mathcal{O}(\epsilon_{\mathrm{mach}})$ ultimately.

\section{A comparison of the CJ-FEAST SVDsolverA and SVDsolverC}
\label{sec: comparison}
We have shown in \Cref{sec: review previous solver} that the CJ-FEAST
SVDsolverC cannot compute the left singular
vectors as accurately as the right singular vectors when associated singular
values are small. As a consequence,
the solver may be numerically backward unstable, that is,
it may fail to converge for a reasonable stopping tolerance in finite precision arithmetic.
In the last section, we have shown that the CJ-FEAST
SVDsolverA can fix this deficiency perfectly. In this section,
we compare the CJ-FEAST SVDsolverA with the CJ-FEAST SVDsolverC in some
detail, and get insight into their efficiency.
Based on the results obtained,
we propose a general-purpose choice strategy between the two solvers
for the robustness and overall efficiency in practical computations.

A core in the two CJ-FEAST SVDsolvers is the construction of two different
approximate spectral projectors.
We focus on the issue of how to choose the series degrees $d$'s,
so that the two different approximate spectral projectors have
the approximately same approximation accuracy and the two solvers
converge at approximately the same rate.
Then based on the costs of one iterations of the two solvers,
for a given stopping tolerance and the interval $[a,b]$ of interest,
we will propose a choice strategy.

In the following, we use the notations hat and tilde to distinguish the two
different functions $l(x)$, $f(x)$ and $\phi_d(l(x))$, etc.,
involved in the CJ-FEAST SVDsolverC and the CJ-FEAST SVDsolverA, respectively.
Concretely, denote by
\begin{displaymath}
    \hat{l}(x)=\frac{2x-\eta^2-\eta_{-}^2}{\eta^2-\eta_{-}^2} \ \mbox{ for }
    x\in [\sigma_{\min}^2,\|A\|^2] \mbox{ and } \ \tilde{l}(x)=\frac{x}{\eta} \
    \mbox{ for } x\in [-\|A\|,\|A\|]
\end{displaymath}
that are used in the CJ-FEAST SVDsolverC and the CJ-FEAST SVDsolverA,
where $\eta$ and $\eta_{-}$ equal $\|A\|$ and $\sigma_{\min}$ or their
estimates, respectively.

For each singular value $\sigma$ of $A$, define
\begin{align*}
    &\hat{\Delta}_{\sigma,a}=|\arccos(\hat l(\sigma^2))-\arccos(\hat l(a^2))|,
    \ \hat{\Delta}_{\sigma,b} = |\arccos(\hat l(\sigma^2))-\arccos(\hat l(b^2))|,
    \\
    &\tilde{\Delta}_{\sigma, a}=|\arccos(\tilde l(\sigma))-\arccos(\tilde
    l(a))|,\  \tilde{\Delta}_{\sigma, b} = |\arccos(\tilde
    l(\sigma))-\arccos(\tilde l(b))|.
\end{align*}
    It is then seen from \Cref{lem:pointwise convergence} that
    the errors $|\hat f(\sigma^2)-\hat{\phi}_d(\hat{l}(\sigma^2))|$
    and $|\tilde f(\sigma)-\tilde{\phi}_d(\tilde{l}(\sigma))|$ are inversely
    proportional to $\hat{\Delta}_{\sigma, a}^4, \hat{\Delta}_{\sigma, b}^4$
    and $\tilde{\Delta}_{\sigma, a}^4, \tilde{\Delta}_{\sigma, b}^4$,
    respectively.

\begin{theorem}\label{thm:twodelta}
    It hold that $\hat{\Delta}_{\sigma, a} \geq
    2\tilde{\Delta}_{\sigma, a}$ and $\hat{\Delta}_{\sigma, b} \geq 2
    \tilde{\Delta}_{\sigma, b}$.
\end{theorem}

\begin{proof}
    Since $\frac{d\arccos(x)}{dx}=\frac{-1}{\sqrt{1-x^2}}$, we have
\begin{align*}
    &\frac{d\arccos(\hat l(x^2))}{dx}=\frac{-\hat l'(x^2)2x}{\sqrt{1-\hat l^2(x^2)}}=\frac{-4x}{(\eta^2-\eta_{-}^2)\sqrt{1-\hat l^2(x^2)}}\\
    &=\frac{-4x}{\sqrt{(\eta^2-\eta_{-}^2)^2-(2x^2-\eta^2-\eta_{-}^2)^2}}
    =\frac{-2x}{\sqrt{(x^2-\eta_{-}^2)(\eta^2-x^2)}}
\end{align*}
and
\begin{displaymath}
    \frac{d\arccos(\tilde l(x))}{dx}=\frac{-\tilde l'(x)}{\sqrt{1-\tilde l^2(x)}}=\frac{-1}{\eta\sqrt{1-\tilde l^2(x)}}=\frac{-1}{\sqrt{\eta^2-x^2}}.
\end{displaymath}
For $x\in (\sigma_{\min}, \|A\|)$, since
\begin{displaymath}
    \frac{-2x}{\sqrt{(x^2-\eta_{-}^2)(\eta^2-x^2)}} < 2 \frac{-1}{\sqrt{\eta^2-x^2}} < 0,
\end{displaymath}
we obtain
\begin{align*}
    &|\arccos(\hat l(\sigma^2))-\arccos(\hat l(a^2))|=\left|\int_{a}^{\sigma}\frac{-2x}{\sqrt{(x^2-\eta_{-}^2)(\eta^2-x^2)}}  dx \right| \\
    &\geq \left|\int_{a}^{\sigma}2\frac{-1}{\sqrt{\eta^2-x^2}} dx \right|
	=2|\arccos(\tilde l(\sigma))-\arccos(\tilde l(a))|.
\end{align*}
Similarly, we obtain
\begin{displaymath}
    |\arccos(\hat l(\sigma^2))-\arccos(\hat l(b^2))| \geq 2|\arccos(\tilde l(\sigma))-\arccos(\tilde l(b))|.
\end{displaymath}
Thus the assertions are proved.
\end{proof}

\begin{remark}\label{rem:2.5 times larger}
    From \Cref{thm:generalconver}, \Cref{thm:priori} and Theorems 5.1--5.2 of
    \cite{jia2022afeastsvdsolver},
    %\Cref{lem:pointwise convergence}, \Cref{accuracyps2}
    %and Theorem 4.1 of \cite{jia2022afeastsvdsolver},
    in order to make the CJ-FEAST SVDsolverA and SVDsolverC converge
    and use approximately the same iterations for a given stopping tolerance,
    we should choose the series degree $d$'s to make the errors of
    $\hat{\phi}_d(\hat{l}(\sigma^2))$ and $\tilde{\phi}_d(\tilde{l}(\sigma))$
    and the accuracy of the corresponding
    approximate spectral projectors
    are approximately equal.
    With such a choice,
    the approximate right singular subspaces of the two SVDsolvers converge roughly at the same speed.
    To this end, we make
    the bound in \eqref{Accuracy of projector2} and the counterpart
    in the CJ-FEAST SVDsolverC equal. As a result, for the series degree
    $d=d_a$ in the CJ-FEAST SVDsolverA and the series degree
    $d=d_c$ in the CJ-FEAST SVDsolverC, we obtain
    \begin{displaymath}
        \frac{\pi^6}{2(d_c+2)^3\min\{\tilde{\Delta}_{\sigma, a}^4,
        \tilde{\Delta}_{\sigma, b}^4\}}
        = \frac{\pi^6}{2(d_a+2)^3
        \min\{\tilde{\Delta}_{\sigma,a}^4,\tilde{\Delta}_{\sigma,b}^4\}},
    \end{displaymath}
    which, by exploiting \Cref{thm:twodelta}, shows that $d_a$ and $d_c$ 
    satisfy
    \begin{equation}\label{dadc}
        d_a\geq 2\sqrt[3]{2}(d_c+2)-2\approx 2.52 d_c+3.
    \end{equation}
\end{remark}

\begin{remark}
    Recall from \Cref{table:Computational cost} that for the same $p$ and $d$,
    %each iteration of the CJ-FEAST SVDsolverC and SVDsolverA cost almost the same.
    the computational cost of one iteration of the CJ--FEAST SVDsolverA is more than that of the CJ--FEAST SVDsolverC.
    Therefore, \Cref{rem:2.5 times larger} means that the CJ-FEAST SVDsolverC is at least $2\sqrt[3]{2}$ times
    as efficient as the CJ-FEAST SVDsolverC when they converge for the same stopping tolerance.
\end{remark}

Next we return to the attainable residual norms by the CJ-FEAST SVDsolverC
in finite precision arithmetic.
Based on the results in \Cref{sec: review previous solver},
to make a Ritz approximation by the CJ-FEAST SVDsolverC converge
for a prescribed tolerance $tol$:
\begin{displaymath}
    \|r\| \leq \|A\| \cdot tol,
\end{displaymath}
relation \eqref{tolfail} shows that a general-purpose {\em smallest} $tol$ should satisfy
\begin{equation}%\label{accuracyC}
    tol\geq \frac{\|A\|}{\sigma}\mathcal{O}(\epsilon_{\mathrm{mach}}).
\end{equation}
Notice that in large SVD computations, one commonly uses
$tol\in [\epsilon_{\mathrm{mach}}^{3/4}, \epsilon_{\mathrm{mach}}^{1/2}]$,
i.e., approximately, $tol\in [10^{-12},10^{-8}]$ with
$\epsilon_{\mathrm{mach}}=2.22\times 10^{-16}$. Therefore,
to make the CJ-FEAST SVDsolverC converge with such a $tol$,
the desired $\sigma$ should meet
\begin{displaymath}
    \frac{\|A\|}{\sigma}\leq \mathcal{O}(\epsilon_{\mathrm{mach}}^{-1/4})
\sim \mathcal{O}(\epsilon_{\mathrm{mach}}^{-1/2});
\end{displaymath}
otherwise, the CJ-FEAST SVDsolverC may fail to converge
in finite precision.

Summarizing the above, we propose a robust choice strategy:
Given $[a,b]$, suppose that there is a $\sigma$ close to $a$
and $\eta$ is an estimate of $\|A\|$ and that we choose
a stopping tolerance
$tol\in [\epsilon_{\mathrm{mach}}^{3/4},\epsilon_{\mathrm{mach}}^{1/2}]$.
Then if $\frac{\eta}{a}\geq \epsilon_{\mathrm{mach}}^{-1/4}$,
the more robust CJ-FEAST SVDsolverA is used;
if not, the more efficient CJ-FEAST SVDsolverC in \cite{jia2022afeastsvdsolver} is used.

\section{Numerical experiments}\label{sec: experiments}
We report numerical experiments to confirm our theory and
illustrate the performance of the CJ-FEAST SVDsolverA and the
CJ-FEAST SVDsolverC.
Our test problems are from The SuiteSparse Matrix Collection
\cite{davis2011university}.
We list some of their basic properties and the interval $[a,b]$ of interest
in \cref{table:Properties of test matrices}.
The exact singular values of $A$ are from \cite{davis2011university}.
Since bounding the singular spectrum of $A$ and estimating the number $n_{sv}$
are not the purpose of this paper,
we will use the known $\eta=\|A\|$, $\eta_{-}=\sigma_{\min}(A)$ and the exact
$n_{sv}$. All the numerical experiments were performed on an Intel Core
i7-9700, CPU 3.0GHz, 8GB RAM using MATLAB R2022b with
$\epsilon_{\mathrm{mach}}=2.22e-16$ under the Microsoft Windows 10 64-bit system.
An approximate singular triplet $(\tilde{\sigma}, \tilde{u}, \tilde{v})$
is claimed to have converged if its relative residual norm attains
the level of $\epsilon_{\mathrm{mach}}$:
\begin{equation}%\label{stopcrit}
   \|r(\tilde{\sigma}, \tilde{u}, \tilde{v})\| \leq \eta \cdot tol=\eta \cdot 1e-14.
\end{equation}

\begin{table}[h]
    \begin{center}
        \resizebox*{\textwidth}{!}{
            \begin{tabular}{|c|c|c|c|c|c|c|c|}
                \hline
                Matrix $A$ & $m$    & $n$   & $nnz(A)$ & $\|A\|$  &$\sigma_{\min}(A)$ & $[a,b]$ & $n_{sv}$\\
                \hline
                rel8            & 345688 & 12347 & 821839   & 18.3  & 0 & $[13, 14]$    & 13       \\
                GL7d12          & 8899   & 1019  & 37519    & 14.4  & 0 & $[11, 12]$   & 17        \\
                flower\_5\_4    & 5226   & 14721 & 43942    & 5.53  & $3.70e-1$ & $[4.1,4.3]$  & 137   \\
                barth5          & 15606  & 15606 & 61484    & 4.23  & $7.22e-11$ & $[1e-8,1e-1]$  & 819 \\
                3elt\_dual      & 9000   & 9000  & 26556    & 3.00  & $6.31e-13$ & $[1e-11,1e-1]$ & 171 \\
                big\_dual       & 30269  & 30269 & 89858    & 3.00  & 0  & $[1e-14,1e-1]$ & 432  \\
                \hline
            \end{tabular}
        }
        \caption{Properties of test matrices, where $nnz(A)$ is the number
        of nonzero entries in $A$, and $\|A\|$, $\sigma_{\min}(A)$ and $n_{sv}$ are from \cite{davis2011university}.}
        \label{table:Properties of test matrices}
    \end{center}
\end{table}

For a practical choice of the series degree $d$,
the results and analysis on the strategies for the CJ-FEAST SVDsolverC in
\cite{jia2022afeastsvdsolver} is
straightforwardly adaptable to the CJ-FEAST SVDsolverA. Precisely, we will
choose
\begin{equation}\label{dchoice}
    d=\left\lceil \frac{D\pi^2}{(\alpha-\beta)^{4/3}} \right\rceil -2
\end{equation}
with $D\in [1, 4]$.
Keep in mind that $d_a$ and $d_c$ denote the series degrees in the 
CJ-FEAST SVDsolverA and SVDsolverC, respectively.
With the same $D$, by \eqref{dadc}, we 
take $d_a=\lceil 2\sqrt[3]{2}d_c\rceil$ 
throughout the experiments.
For the subspace dimension $p$,
we will take $p=\lceil \mu n_{sv}\rceil $ with $\mu \in [1.1, 1.5]$.

\subsection{Computing singular triplets with not small singular values}
We apply \cref{alg:Augmented-PSVD} and the CJ-FEAST SVDsolverC
to GL7d12, whose desired singular values $\sigma$ are not small:
$\|A\|/\sigma=\mathcal{O}(1)$.
In terms of \eqref{dadc} and
\eqref{dchoice}, we take $D = 4$ to obtain the polynomial degree
$d_a=698$ and $d_c=276$,
and take $p=\lceil 1.2\times 17\rceil=21$.
It is observed that the two solvers converged at roughly the same
iteration steps $k_a=6$ and $k_c=7$, respectively.
Then we take $D = 2$ to obtain
$d_a=348$ and $d_c=137$,
and take $p=\lceil 1.5\times 17\rceil=26$.
They are found to have converged at
roughly the same iteration steps $k_a=7$ and $k_c=9$, respectively.
We have also taken some other $d_a$ and $d_c$ with the same $D$,
and the same $p>n_{sv}$, and observed that the two solvers used almost the
same iterations to achieve $tol=1e-14$. In
\cref{fig:moderate singular values GL7d12},
we draw the convergence processes of the two solvers for the
singular triplet with $\sigma=11.844206301985537$.

\begin{figure}[tbhp]
    \centering
    \subfloat[CJ-FEAST SVDsolverA, $d=698, p=21$.]{\label{fig:GL7d12_aug1184}\includegraphics[scale=0.4]{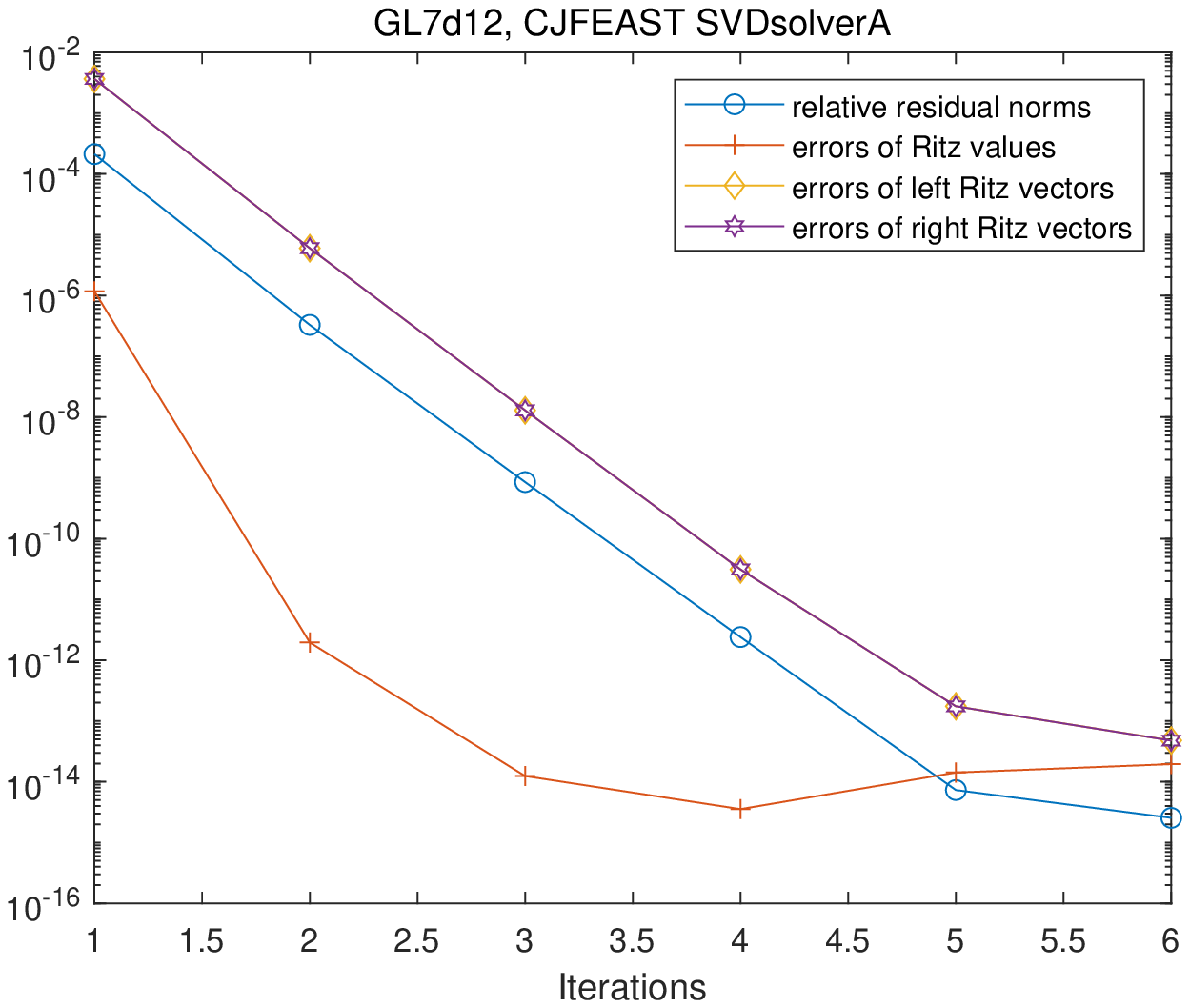}}
    \subfloat[CJ-FEAST SVDsolverC, $d=276, p=21$.]{\label{fig:GL7d12_crossproduct1184}\includegraphics[scale=0.4]{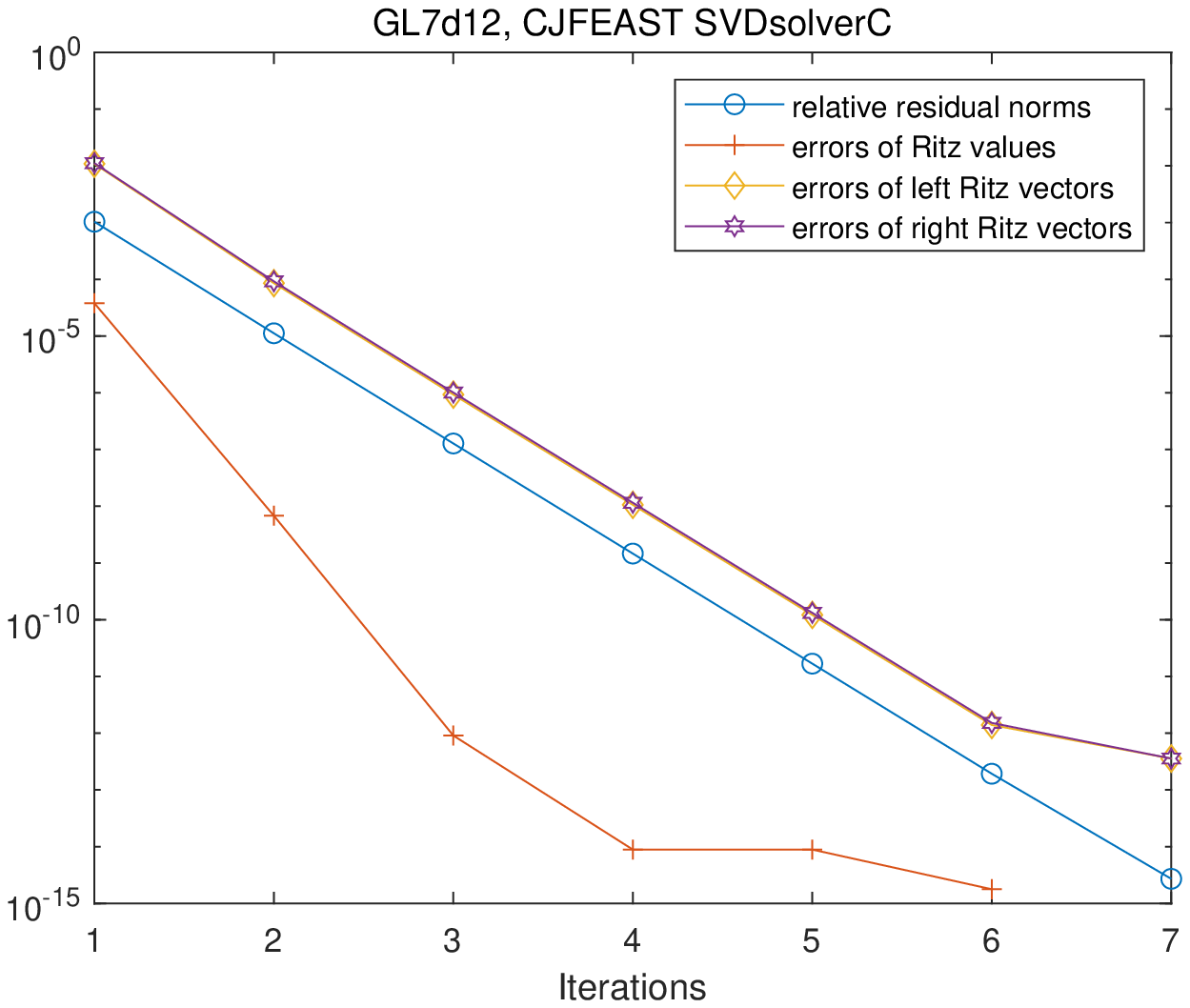}}

    \subfloat[CJ-FEAST SVDsolverA, $d=348, p=26$.]{\label{fig:GL7d12_aug1183}\includegraphics[scale=0.4]{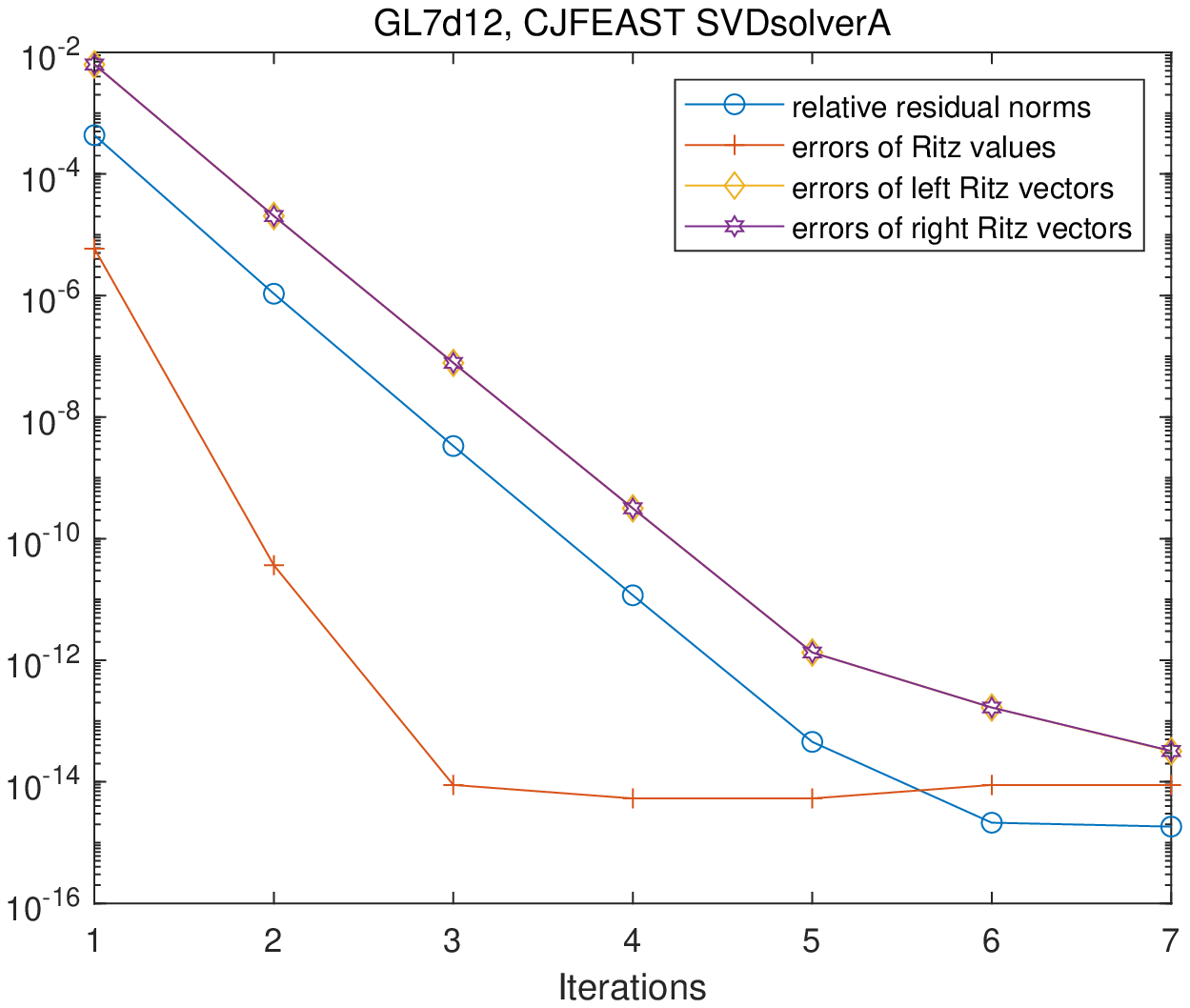}}
    \subfloat[CJ-FEAST SVDsolverC, $d=137, p=26$.]{\label{fig:GL7d12_crossproduct1183}\includegraphics[scale=0.4]{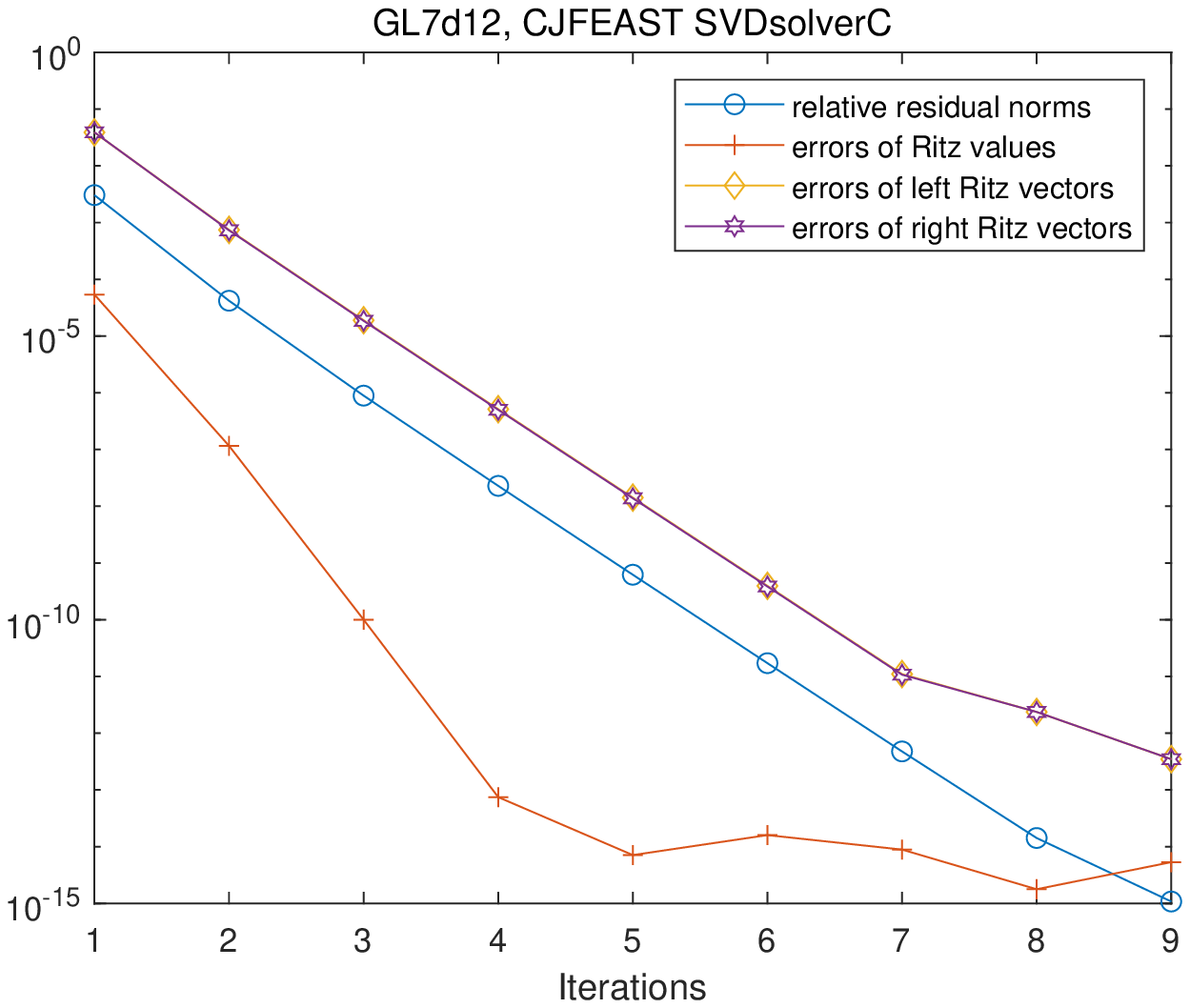}}
    \caption{Convergence processes of approximate singular triplets of GL7d12.}
    \label{fig:moderate singular values GL7d12}
\end{figure}

For flower\_5\_4, we take $D = 2$ to obtain
$d_a=928$ and $d_c=365$,
and take $p=\lceil 1.2\times 137\rceil=165$.
The two SVDsolvers
converged at iteration steps $k_a=8$ and $k_c=11$.
For rel8, we take $D = 2$ to obtain
$d_a = 561$ and
$d_c = 222$,
and $p=\lceil 1.1\times 13\rceil=15$.
The two SVDsolvers converged at iteration steps $k_a=17$
and $k_c=20$ separately,
roughly the same. In \cref{fig:moderate singular values},
we depict the convergence processes of the two solvers for
computing the
singular triplet with $\sigma=4.299030932949072$ of flower\_5\_4 and
$\sigma=13.984665903216351$ of rel8.

\begin{figure}[tbhp]
    \centering
    \subfloat[flower\_5\_4, CJ-FEAST SVDsolverA]{\includegraphics[scale=0.4]{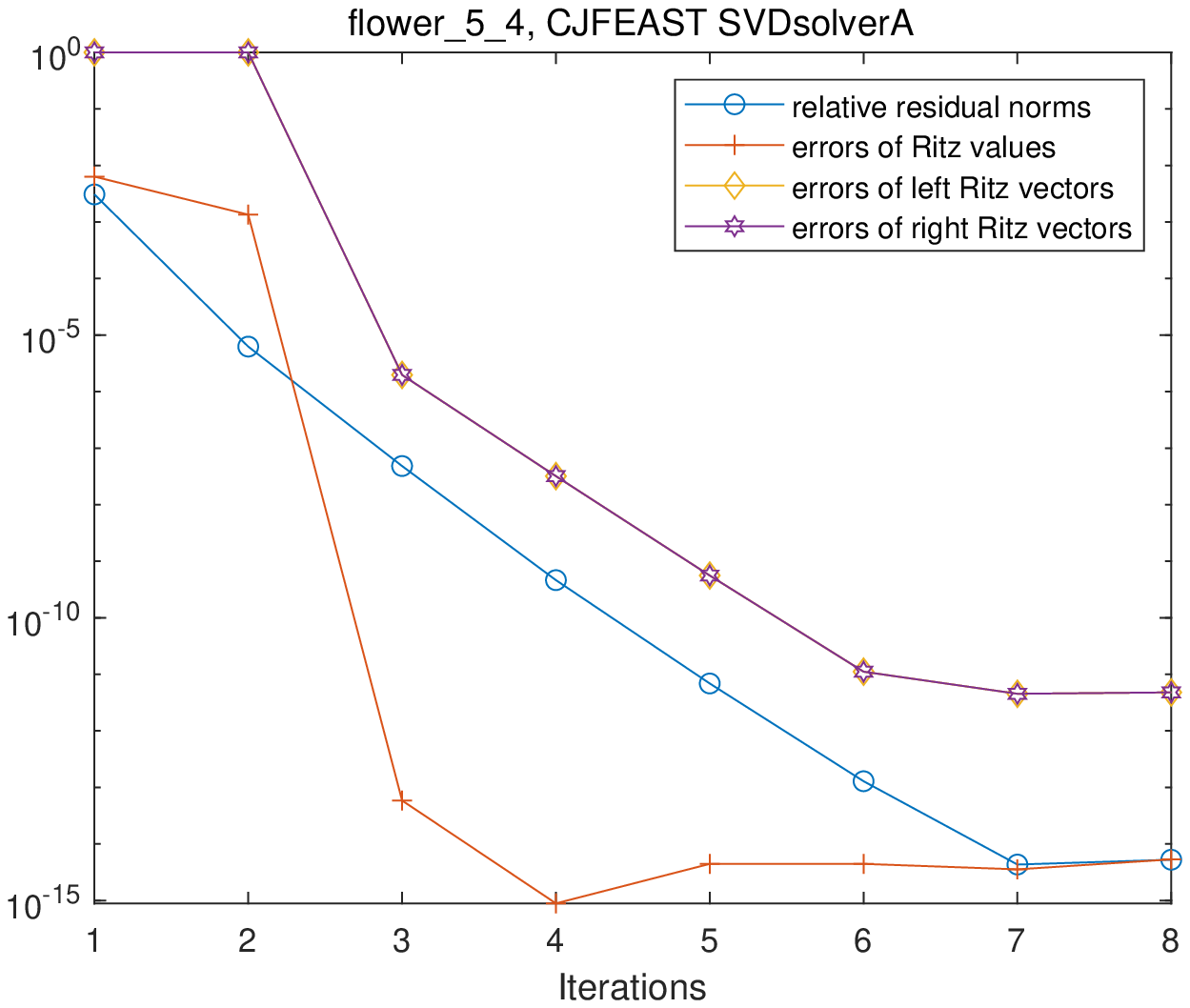}}
    \subfloat[flower\_5\_4, CJ-FEAST SVDsolverC]{\includegraphics[scale=0.4]{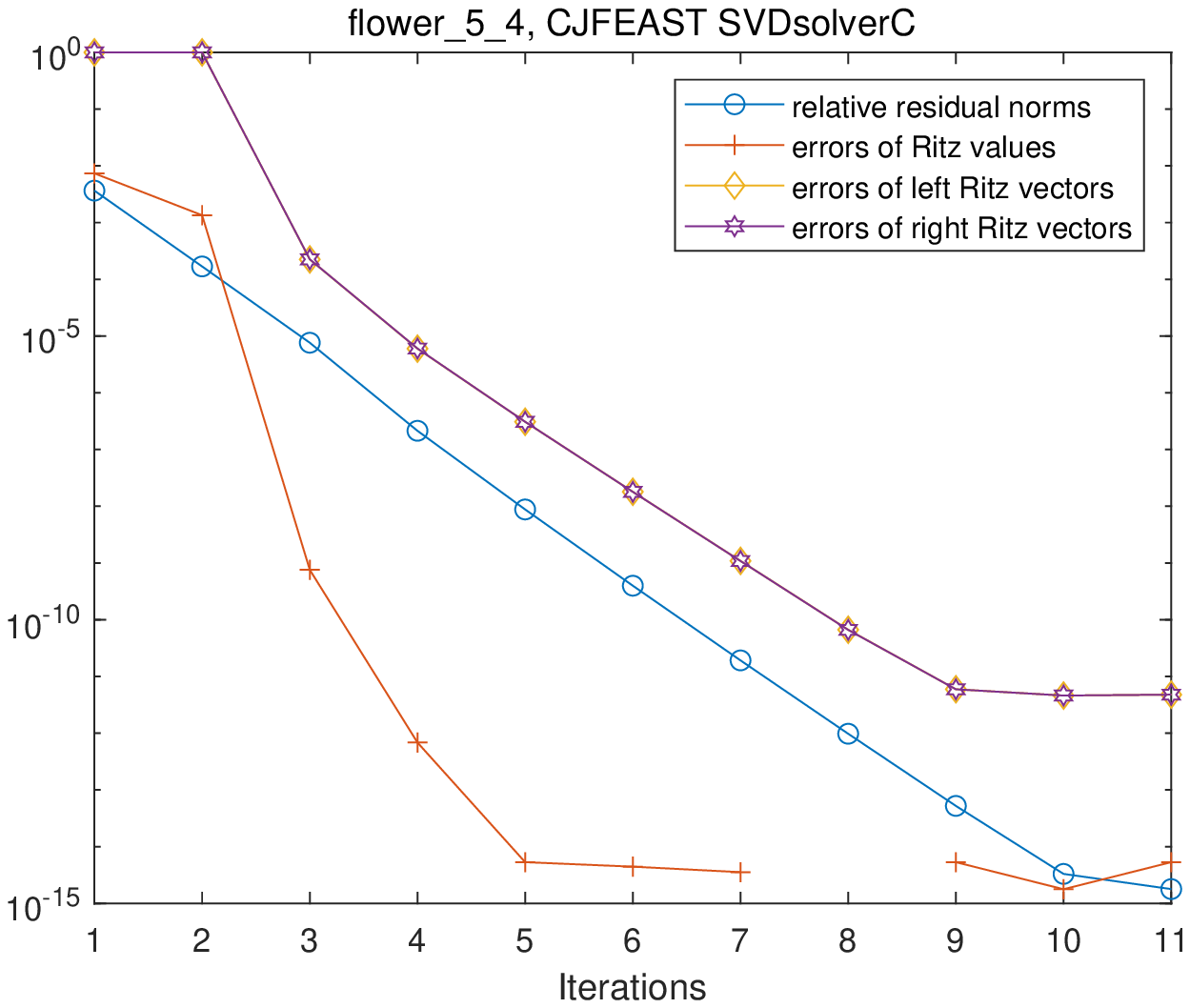}}

    \subfloat[rel8, CJ-FEAST SVDsolverA]{\includegraphics[scale=0.4]{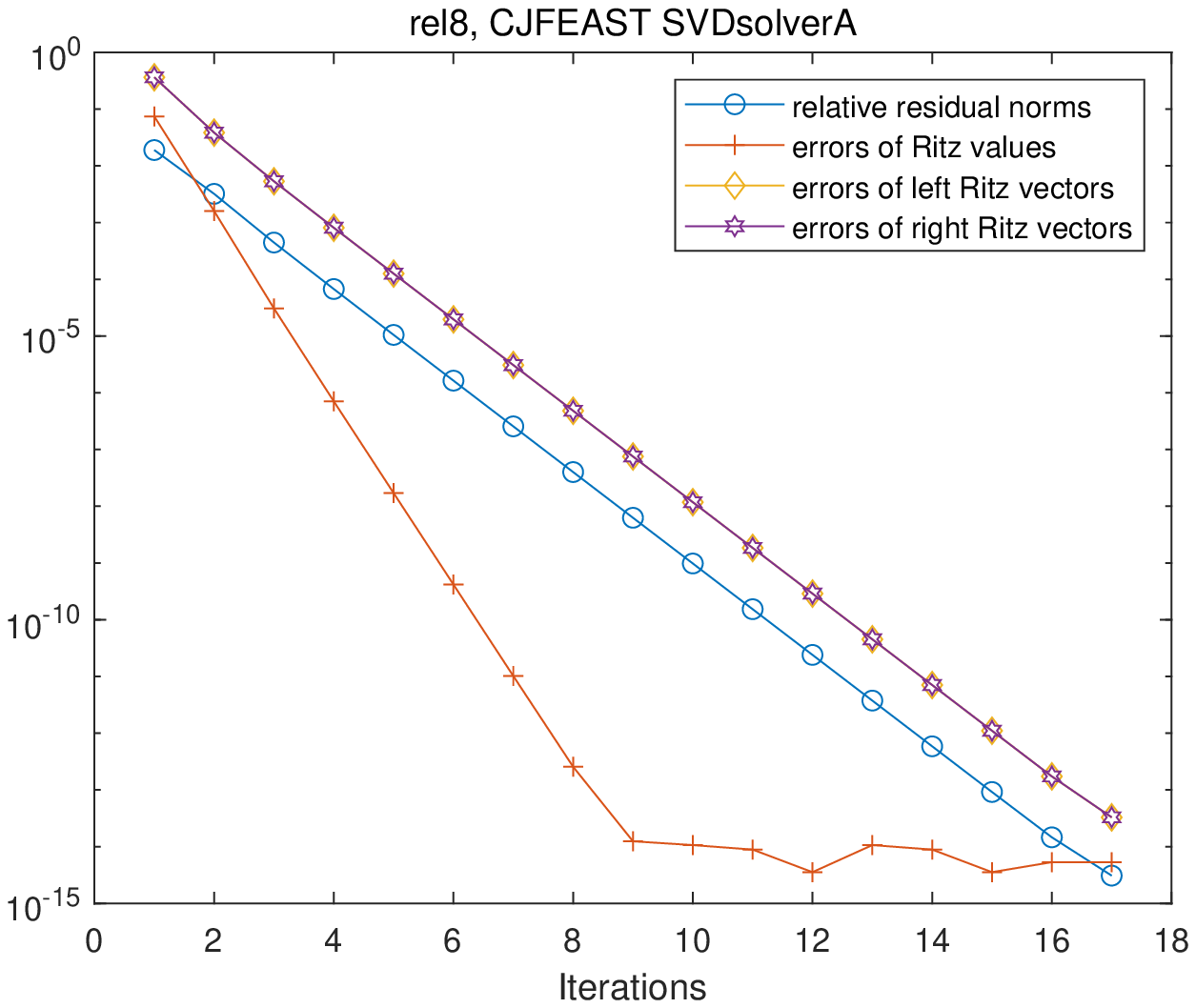}}
    \subfloat[rel8, CJ-FEAST SVDsolverC]{\includegraphics[scale=0.4]{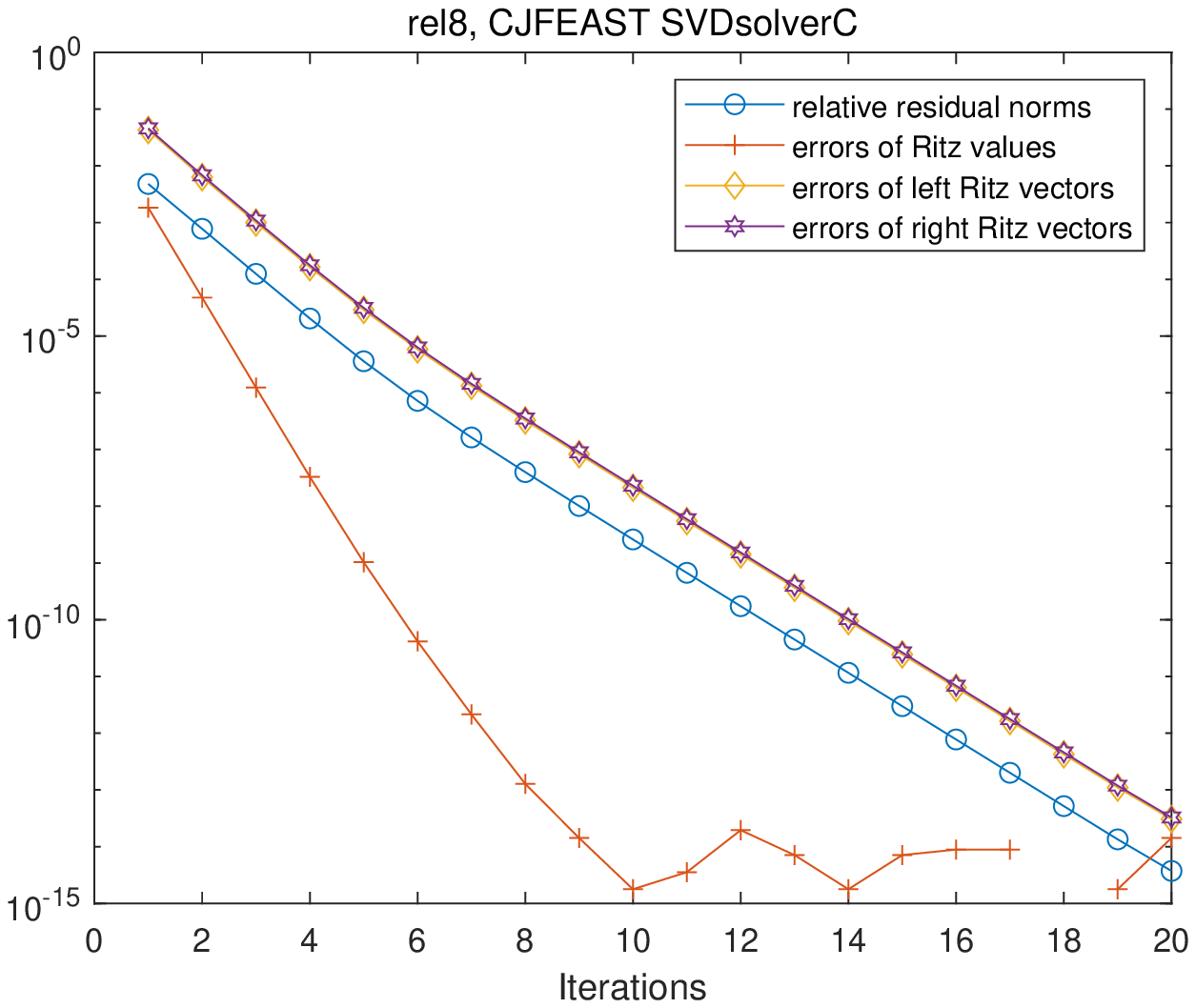}}
    \caption{Convergence processes of approximate singular triplets for not small singular values.}
    \label{fig:moderate singular values}
\end{figure}

These experiments justify that the choice strategy \eqref{dchoice}
of the series degree $d$ works well and, meanwhile,
they confirm \cref{rem:2.5 times larger}.
Clearly, we see from \Cref{fig:moderate singular values GL7d12} and
\cref{fig:moderate singular values}
that the convergence processes of the two solvers are very similar and the
Ritz value and the corresponding left and right Ritz vectors
have very comparable accuracy at each iteration. These confirm
that the CJ-FEAST SVDsolverC and SVDsolverA can compute the singular
triplets accurately when the desired singular values are not small but
the former more efficient than the latter. We can also
find that the errors of Ritz values are approximately
squares of those of the left and right Ritz vectors as well as
residual norms until the Ritz values have converged with the full accuracy
$\|A\|\mathcal{O}(\epsilon_{\mathrm{mach}})$,
as the results in \Cref{sec: review previous solver} and
\Cref{thm:priori} indicate.

\subsection{Computing singular triplets with small singular values}
We apply \cref{alg:Augmented-PSVD} and the CJ-FEAST SVDsolverC to barth5,
3elt\_dual and big\_dual.
For each problem, at least one of the desired singular values is small.

For barth5, one of the desired singular values is $\sigma=1.1050e-8$.
We take $D = 1$ to obtain
$d_a=1453$ and
$d_c=576$,
and the subspace dimension $p=\lceil 1.2\times 819\rceil=983$.
We run $10$ iterations, and draw their convergence processes
in \cref{fig:small singular values} (a) and (b).

For 3elt\_dual, one of the desired singular values is $\sigma=6.8890e-11$.
%According to \eqref{dadc} and \eqref{dchoice},
We take $D = 1$ to obtain
$d_a=918$ and
$d_c=364$,
and $p=\lceil 1.1\times 171\rceil=189$.
We run $15$ iterations, and draw their convergence processes
in \cref{fig:small singular values} (c) and (d).

For big\_dual, one of the desired singular values is $\sigma=8.7726e-13$.
We take $D = 1$ in \eqref{dchoice} to obtain
$d_a=918$ and
$d_c=591$,
and $p=\lceil 1.2\times 432\rceil=519$.
We run $10$ iterations, and draw their convergence processes
in \cref{fig:small singular values} (e) and (f).

\begin{figure}[tbhp]
    \centering
    \subfloat[barth5, CJ-FEAST SVDsolverA]{\label{fig:barth5_aug}\includegraphics[scale=0.4]{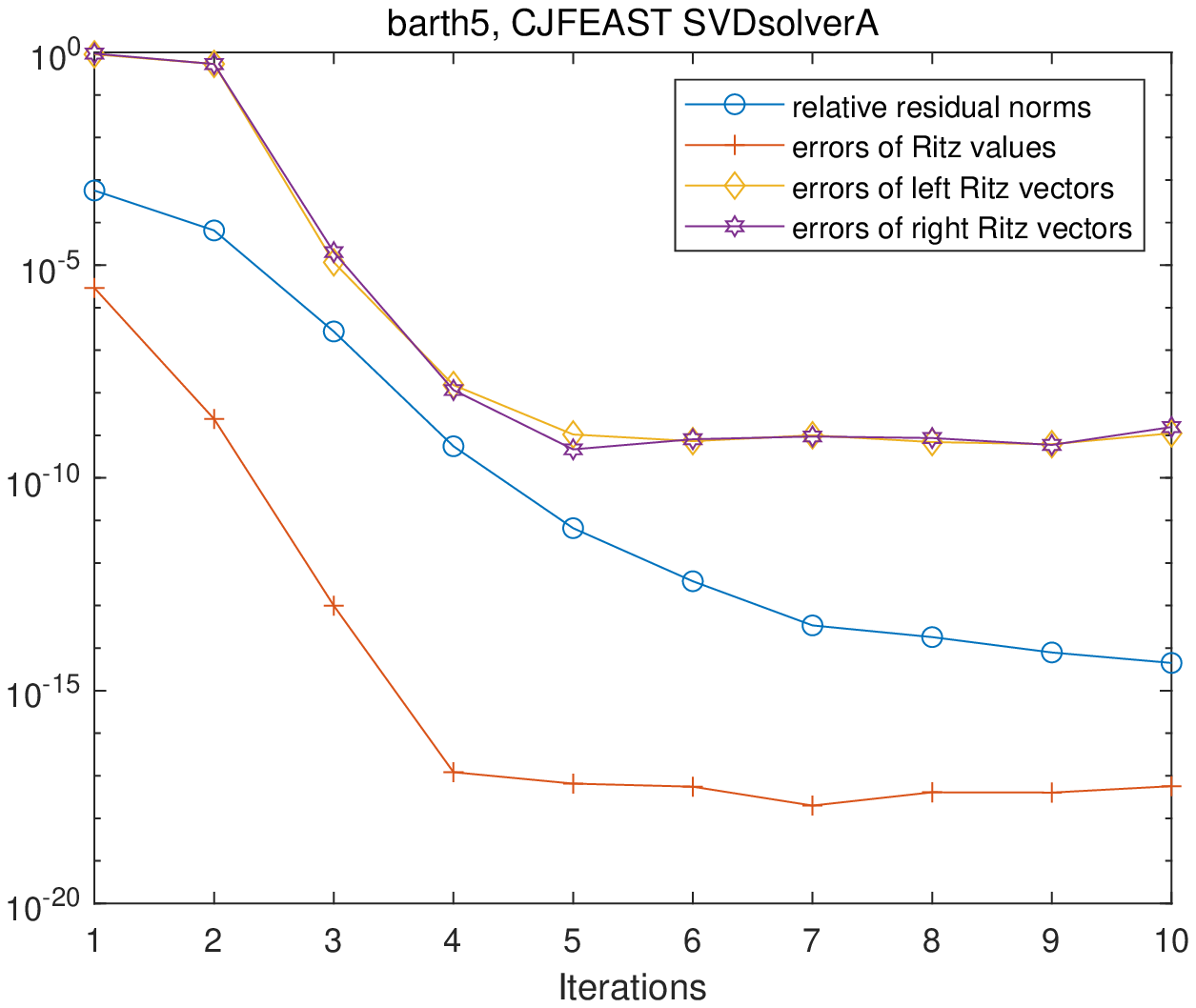}}
    \subfloat[barth5, CJ-FEAST SVDsolverC]{\label{fig:barth5_crossproduct}\includegraphics[scale=0.4]{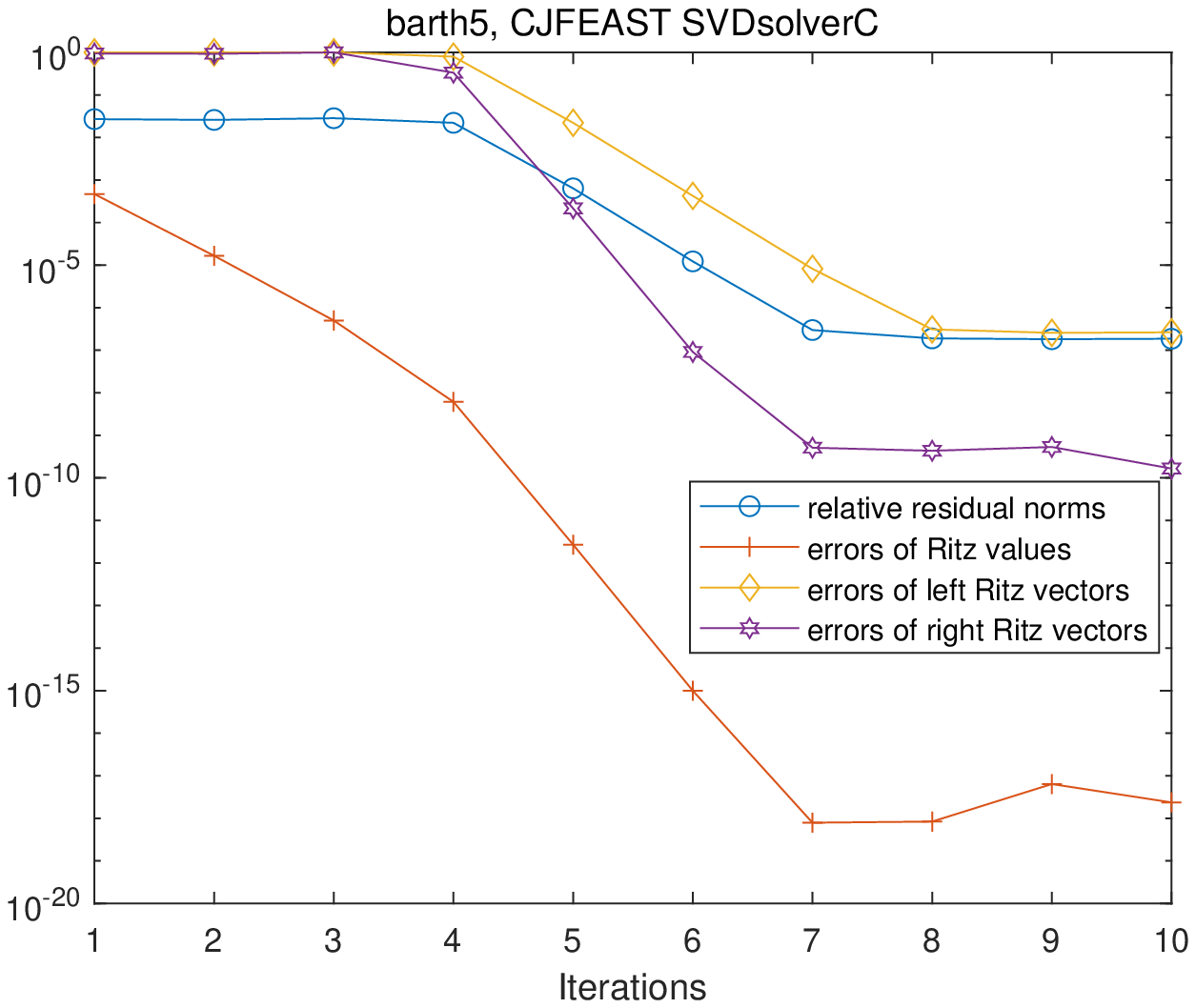}}

    \subfloat[3elt\_dual, CJ-FEAST SVDsolverA]{\includegraphics[scale=0.4]{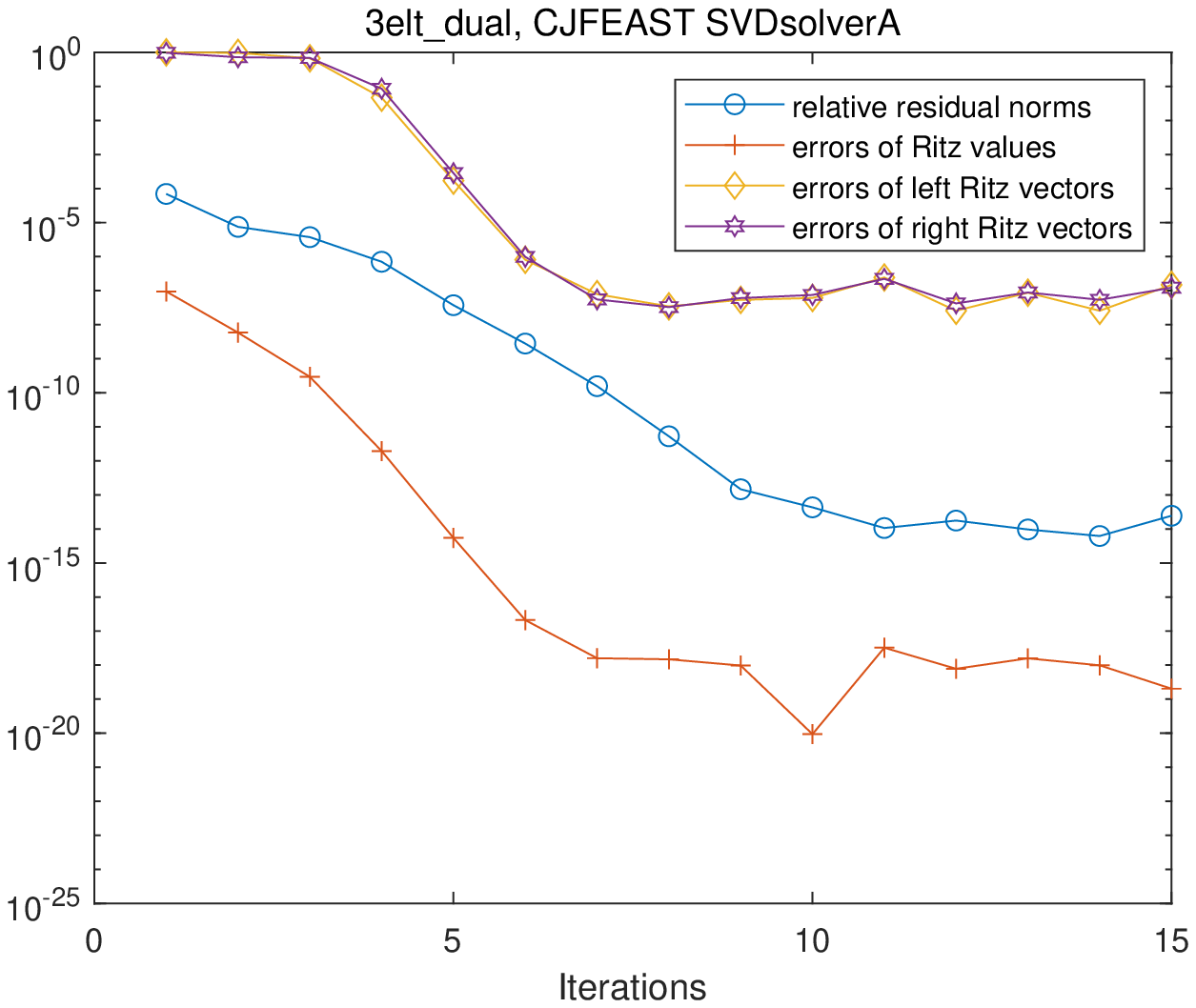}}
    \subfloat[3elt\_dual, CJ-FEAST SVDsolverC]{\includegraphics[scale=0.4]{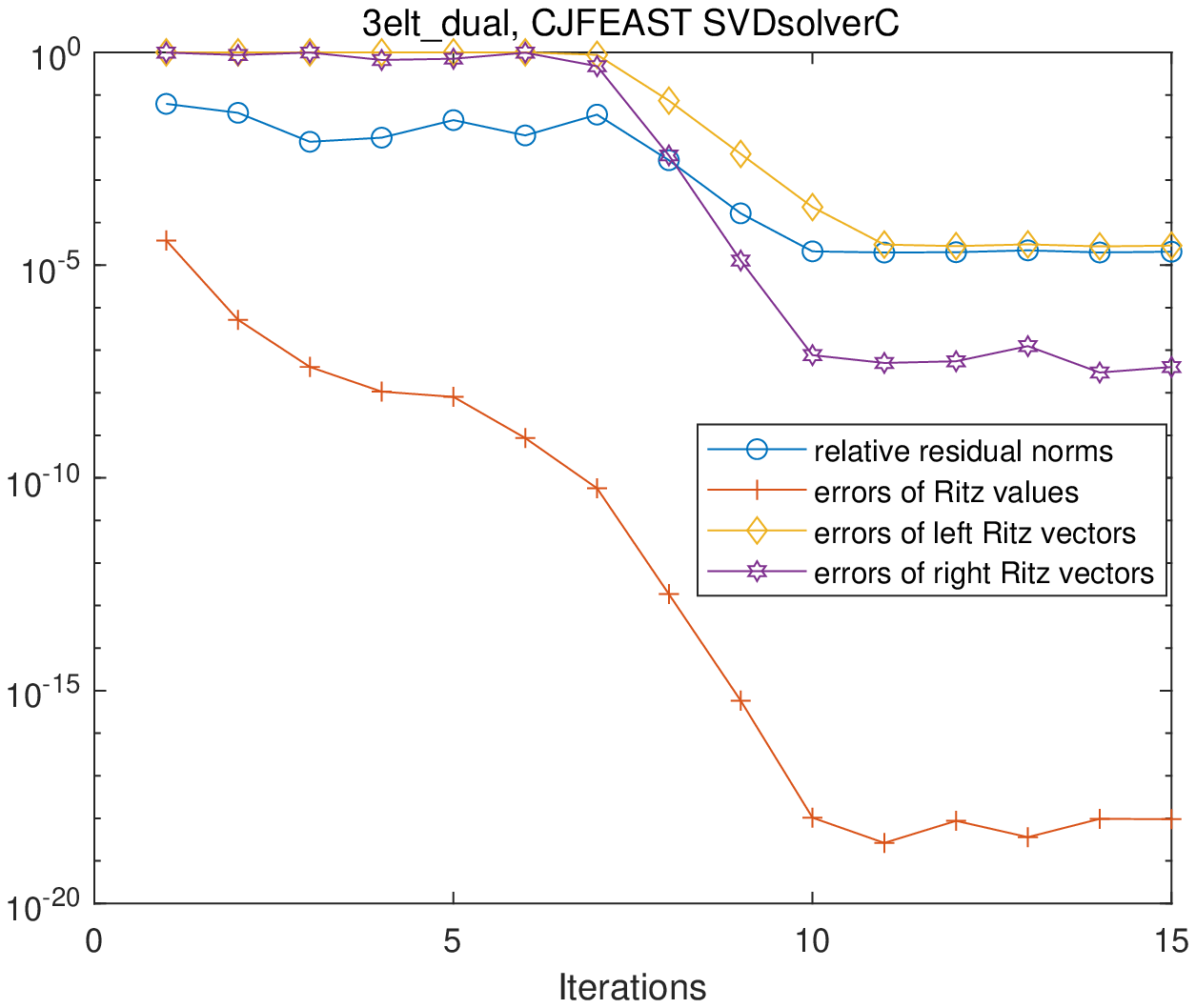}}

    \subfloat[big\_dual, CJ-FEAST SVDsolverA]{\includegraphics[scale=0.4]{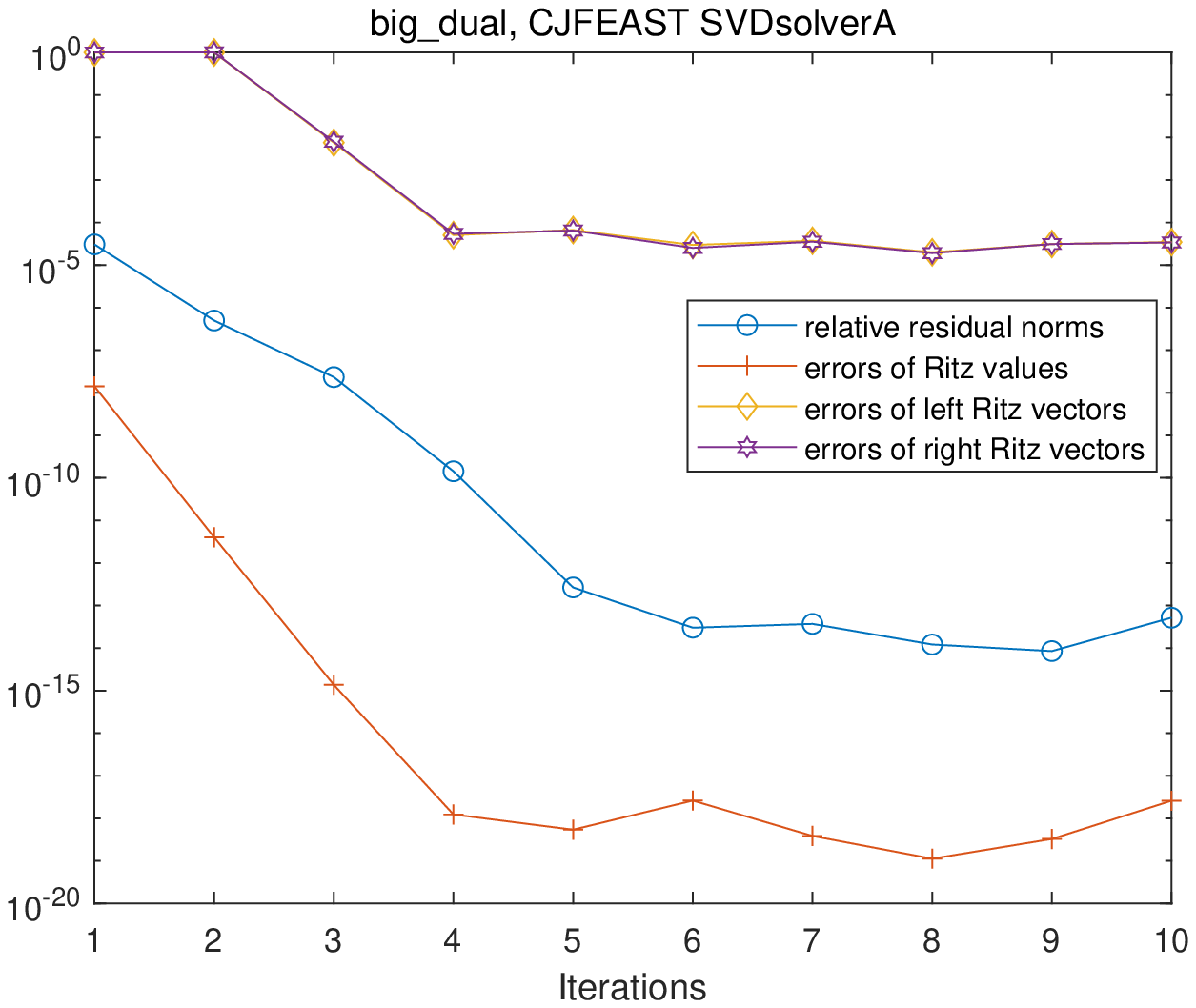}}
    \subfloat[big\_dual, CJ-FEAST SVDsolverC]{\includegraphics[scale=0.4]{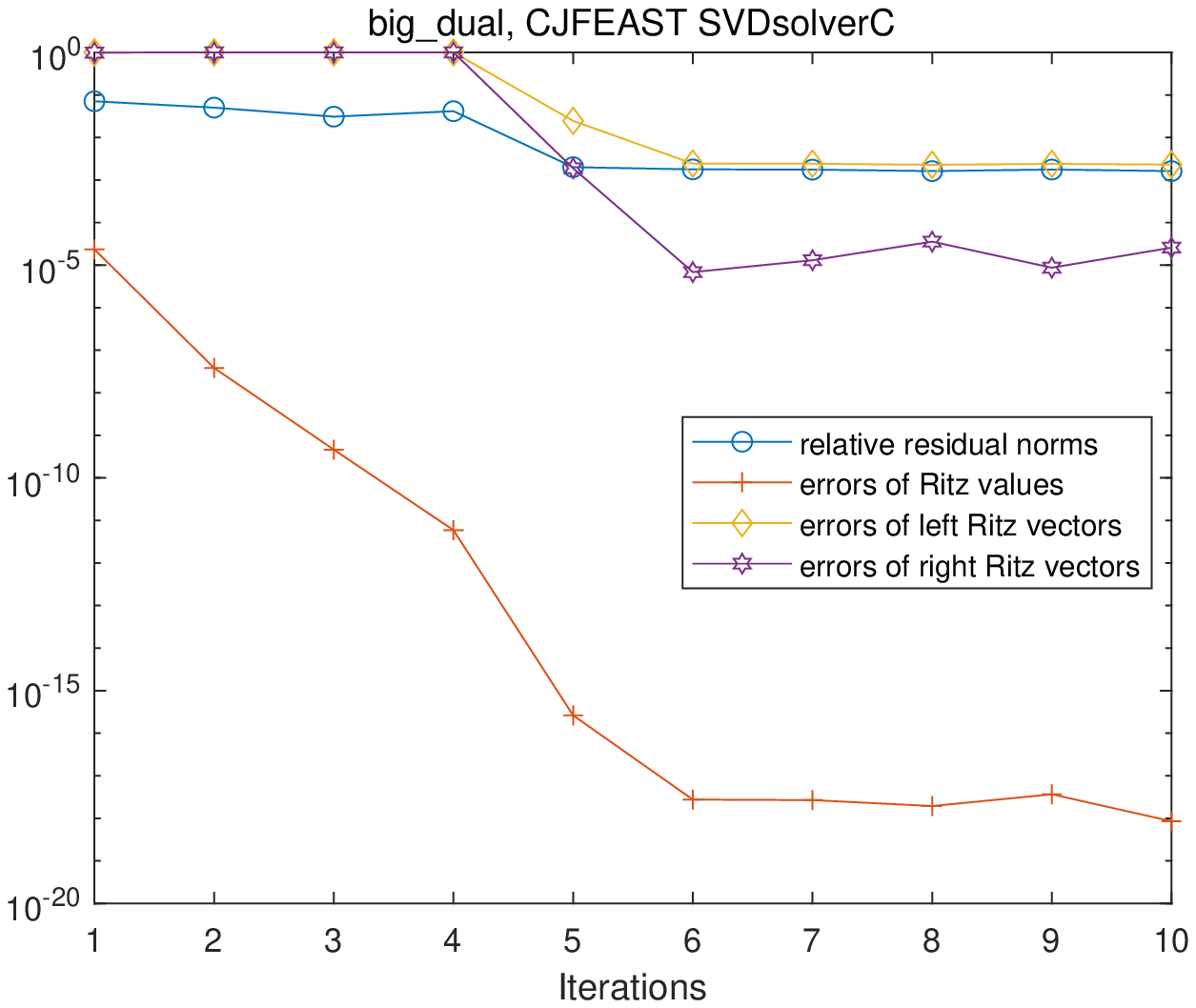}}
    \caption{Convergence processes of approximate singular triplets for small singular values.}
    \label{fig:small singular values}
\end{figure}

Several comments are made on \cref{fig:small singular values}.
First, for each problem,
the left and right Ritz vectors by the CJ-FEAST SVDsolverA
always have similar accuracy at the same iteration.
Second,
the right Ritz vectors computed by the two SVDsolvers have similar accuracy,
but the errors of the left Ritz vectors computed by the
CJ-FEAST SVDsolverC are a few orders larger than those computed by the CJ-FEAST SVDsolverA.
Third, as expected, the relative residual norms of the Ritz approximation
by the CJ-FEAST SVDsolverA decrease to
$\mathcal{O}(\epsilon_{\mathrm{mach}})$,
but those by the CJ-FEAST SVDsolverC stagnate before achieving
$\mathcal{O}(\epsilon_{\mathrm{mach}})$
due to the much less accurate left Ritz vectors. In fact,
for barth5, 3elt\_dual and big\_dual, the ultimately
relative residual norms are approximately $1e-7, 1e-5$
and $1e-3$, respectively, which are precisely $\|A\|/\sigma$
times larger than $\mathcal{O}(\epsilon_{\mathrm{mach}})$.
These facts justify our results and
analysis in \Cref{sec: review previous solver} and \Cref{sec: comparison},
and demonstrate that the CJ-FEAST SVDsolverC fails to converge
in finite precision arithmetic when \eqref{tolfail} is violated.
Fourth, the final errors of the Ritz values by the two solvers are
$\|A\|\mathcal{O}(\epsilon_{\mathrm{mach}})$, meaning that they
compute the singular values $\sigma$ to working precision,
independently of the size of $\sigma$.

In summary, the numerical experiments have illustrated
that the CJ-FEAST SVDsolverC may not compute left
singular vectors as accurately as the right ones and
may not make the residual norm drop below a reasonable $tol$
when at least one desired singular value is small. It is
conditionally numerically backward stable, but
the CJ-FEAST SVDsolverA is always unconditionally
numerically backward stable.

\section{Conclusions}\label{sec: conclusion}
Based on the convergence results on the CJ-FEAST SVDsolverC,
we have made an in-depth analysis of the
numerical backward stability of the solver and proved
that it may be numerically backward unstable
in finite precision arithmetic when computing small singular triplets.
The reason is that
it may compute the associated left singular vector much less accurately than
the right singular vector. Consequently, the residual norms of
Ritz approximations may not decrease to a reasonably prescribed
tolerance and the solver may thus fail in finite precision arithmetic when
$\|A\|/\sigma$ is large.

As an alternative, we have proposed
an augmented matrix $S_A$ based CJ-FEAST SVDsolverA. It first constructs
an approximate spectral projector $P$ of $S_A$ associated
with all the eigenvalues $\sigma\in [a,b]$ by exploiting the
Chebyshev--Jackson series expansion, then performs subspace iteration on $P$
to construct left and right searching subspaces independently, and finally
computes the Ritz approximations of the desired singular triplets with respect
to the left and right subspaces.

We have derived estimates for the eigenvalues of $P$
and the approximation error $\|P_{S_A}-P\|$
in terms of the series degree $d$.
We have established convergence results on the approximate left and right
singular subspaces
and the Ritz approximations, and shown that
the left and right Ritz vectors computed by the CJ-FEAST SVDsolverA
always have similar accuracy,
no matter how small the desired singular values are. We have proved that
the ultimate relative residual norms of Ritz approximations can always
attain $\mathcal{O}(\epsilon_{\mathrm{mach}})$,
meaning that the solver is numerically backward stable in finite precision
arithmetic. Therefore, the CJ-FEAST SVDsolverA is more robust than the CJ-FEAST
SVDsolverC when $\|A\|/\sigma$ is large.
We have made a theoretical comparison of the CJ-FEAST SVDsolverA
and SVDsolverC, showing that the latter is at least
$2\sqrt[3]{2}$ times as efficient as the former if they both converge for the
same tolerance $tol$.
Therefore, the CJ-FEAST SVDsolverC and SVDsolverA have their own
merits. For the purpose of robustness and overall efficiency, we have
proposed a practical choice strategy between the two CJ-FEAST SVDsolvers.

Illuminating numerical experiments have justified all of our
results.

\section*{Declarations}

%{\bf Conflict of interest} \
The two authors declare that they have no
financial interests, and they read and approved the final manuscript.
The algorithmic Matlab code is available upon reasonable request from
the corresponding author.

%\bibliographystyle{siamplain}
%\bibliography{references}

\end{document}